\documentclass[11pt]{article}
\usepackage[numbers,square]{natbib}
\usepackage{amsmath}
\usepackage{amsthm}
\usepackage{amsfonts}
\usepackage{amssymb}
\usepackage{siunitx}
\usepackage{commath}
\usepackage{graphicx}
\usepackage{xspace}
\usepackage{color}
\usepackage[lofdepth,lotdepth]{subfig}
\usepackage{stmaryrd}
\usepackage{url}
\usepackage{amsthm}
\usepackage{graphbox}
\usepackage{footnote}
\usepackage{multirow}
\makesavenoteenv{tabular}
\makesavenoteenv{table}
\usepackage[hidelinks]{hyperref}
\hypersetup{colorlinks = true, urlcolor = blue,
  linkcolor = blue, citecolor = blue}
\usepackage{cleveref}

\usepackage[margin=0.7in]{geometry}
\makeatletter
\newcommand{\tnorm}{\@ifstar\@tnorms\@tnorm}
\newcommand{\@tnorms}[1]{%
  \left|\mkern-1.5mu\left|\mkern-1.5mu\left|
   #1
  \right|\mkern-1.5mu\right|\mkern-1.5mu\right|
}
\newcommand{\@tnorm}[2][]{%
  \mathopen{#1|\mkern-1.5mu#1|\mkern-1.5mu#1|}
  #2
  \mathclose{#1|\mkern-1.5mu#1|\mkern-1.5mu#1|}
}
\makeatother

\newtheorem{theorem}{Theorem}
\newtheorem{lemma}{Lemma}

\title{Preconditioning of a hybridizable discontinuous Galerkin method
  for the coupled Stokes--Darcy system}
\author{
  E. Henr\'iquez\thanks{Department of Applied Mathematics, University of
    Waterloo, ON, Canada (\url{ehenriqu@uwaterloo.ca}),
    \url{https://orcid.org/0000-0002-0243-0368}}
  \and
  M. Kuchta\thanks{Simula Research Laboratory, Oslo, Norway
    (\url{miroslav@simula.no}),
    \url{https://orcid.org/0000-0002-3832-0988}}
  \and
  J. J. Lee\thanks{Department of Mathematics, Baylor University,
    TX, USA (\url{jeonghun_lee@baylor.edu}),
    \url{https://orcid.org/0000-0001-5201-8526}}
  \and
  S. Rhebergen\thanks{Department of Applied Mathematics, University of
    Waterloo, ON, Canada (\url{srheberg@uwaterloo.ca}),
    \url{https://orcid.org/0000-0001-6036-0356}}}
\begin{document}
\maketitle
\begin{abstract}
  We propose parameter-robust preconditioners for the statically
  condensed linear system arising from a hybridizable discontinuous
  Galerkin discretization of the coupled Stokes--Darcy system. The
  design strategy relies on first applying the
  operator-preconditioning framework [\emph{Numer. Linear Algebra
    Appl.}, 18(1):1--40, 2011] to construct a preconditioner for the
  non-condensed discretization. This is done by proving uniform
  well-posedness of the scheme. Next, we prove robustness of the
  resulting condensed preconditioner applied to the reduced linear
  system using the framework we proposed in [\emph{SIAM
    J. Sci. Comput.}, 47(6):A3212--A3238, 2025]. Numerical examples
  demonstrate robustness of the proposed preconditioners.
\end{abstract}
\section{Introduction}
\label{s:introduction}

The coupled Stokes--Darcy equations, with Beavers--Joseph--Saffman
interface conditions \cite{beavers1967boundary,saffman1971boundary},
describe the interaction between free flow and porous media
flow. Various finite element discretizations have been introduced to
discretize this coupled system of partial differential
equations. These include a continuous Galerkin locally conservative
discretization \cite{riviere2005locally}, fully mixed formulations
\cite{camano2015new,gatica2011analysis}, a discretization that imposes
the interface coupling terms strongly in the discrete velocity space
\cite{marquez2015strong}, stabilized Crouzeix--Raviart element
formulations \cite{burman2005stabilized}, and a strongly conservative
discontinuous Galerkin method \cite{kanschat2010strongly}. In this
paper, however, we discretize the coupled Stokes--Darcy problem by a
hybridizable discontinuous Galerkin (HDG) method.

Various HDG discretizations have been introduced for the coupled
Stokes--Darcy problem. These include a fully mixed HDG method with
stress, vorticity, velocity, and velocity trace as Stokes unknowns and
velocity, pressure, and pressure trace for Darcy
\cite{gatica2017analysis}; a divergence-conforming HDG method for
Stokes combined with a mixed method for Darcy \cite{fu2018strongly};
and stabilized hybrid formulations using vector and scalar multipliers
for the Stokes and Darcy velocities, respectively
\cite{igreja2018stabilized,igreja2025scalar}. We will discretize the
coupled Stokes--Darcy problem using the HDG discretization proposed
and analyzed in \cite{cesmelioglu2020embedded}. This discretization
imposes mass conservation pointwise on the cells and is
$H(\text{div})$-conforming on the whole domain, i.e., the
discretization is strongly conservative \cite{kanschat2010strongly}.

The HDG method was introduced by Cockburn et al.
\cite{cockburn2009hybridizable} as a discontinuous Galerkin method
that allows for static condensation. With static condensation the
(local) cell degrees-of-freedom are eliminated from the HDG
discretization. The resulting reduced system is one for (global) face
unknowns only. Once the reduced problem has been solved, the cell
unknowns are recovered inexpensively by local postprocessing. The most
expensive step in the static condensation process is solving for the
face unknowns. While parameter-robust preconditioners have been
designed for the reduced problem resulting from HDG discretizations of
the Stokes and Darcy problems separately (see
\cite{henriquez2025parameter}), we are not aware of any
parameter-robust preconditioners for HDG discretizations of the
coupled Stokes--Darcy problem. This is the objective of the present
work.

The operator-mapping framework (see \cite{mardal2011preconditioning}
for a review) has proven to be an effective approach to determine
parameter-robust preconditioners for symmetric problems. By performing
a well-posedness analysis in appropriately weighted norms that depend
on the problem parameters, one can establish uniform well-posedness of
the discretization, which in turn induces block-diagonal
preconditioners given by the matrix representations of the norms used
in the analysis. This framework has been extensively applied to
problems such as reaction--diffusion, Stokes, Darcy, time-dependent
Stokes, the Reissner--Mindlin plate model, and optimal control
problems (see \cite{mardal2011preconditioning} and the references
therein).

The operator-mapping framework has also been applied to determine
parameter-robust preconditioners for (continuous Galerkin)
discretizations of the coupled Stokes--Darcy problem. For instance, in
\cite{holter2020robust} the authors extend the operator-mapping
preconditioning framework to multiphysics problems using a formulation
in which a Lagrange multiplier is used to couple flow problems across
the interface \cite{layton2002coupling}. In \cite{holter2021robust},
parameter-robust preconditioners are developed for the Stokes--Darcy
problem with the Darcy equation in primal form, again employing an
interface Lagrange multiplier. In \cite{boon2022robust}, the authors
present preconditioners for three different formulations and their
associated discretizations, using conforming and nonconforming finite
element methods as well as finite volume methods, with the Darcy
equation in primal form \cite{discacciati2002mathematical}. More
recently, parameter-robust preconditioners for a Lagrange-multiplier
formulation have been introduced in \cite{hu2025parameter} that rely
on fractional operators on the interface. By imposing the interface
conditions directly in the finite element polynomial space, the
authors in \cite{boon2025parameter} are able to determine
parameter-robust preconditioners that avoid fractional operators.

In this work, we apply the framework recently introduced by the
authors in \cite{henriquez2025parameter}. This framework constitutes
an extension of the operator-preconditioning framework specifically
for hybridized discretizations. The construction of preconditioners
for hybridized discretizations proceeds in two steps:
\begin{enumerate}
\item A preconditioner for the non-condensed system is obtained by
  establishing uniform well-posedness with respect to the physical and
  discretization parameters. This step follows the usual
  operator-mapping framework \cite{mardal2011preconditioning}.
\item All cell degrees-of-freedom are eliminated from the
  preconditioner from step 1 resulting in a preconditioner for the
  reduced problem. To determine whether this reduced preconditioner is
  a parameter-robust preconditioner for the reduced problem one needs
  to verify \cref{eq:robust-schur-precon-1}.
\end{enumerate}
This approach has successfully been applied to IP-HDG discretizations
of the stationary and time-dependent Stokes equations
\cite{henriquez2025parameter,henriquez2026robust}, a hybrid-BDM
discretization of the Darcy equations with a reaction term
\cite{henriquez2025parameter}, and a four-field HDG discretization of
Biot’s consolidation model \cite{henriquez2025preconditioning}. In
addition to presenting the first block preconditioners for a
hybridized discretization of the coupled Stokes--Darcy problem, we
also emphasize that our preconditioners neither rely on fractional
operators nor require enforcing the interface conditions in the finite
element polynomial spaces, thereby simplifying their implementation.

The remainder of this paper is organized as follows. In
\Cref{sec:stokes-darcy-sys} we present the coupled Stokes--Darcy
problem and in \Cref{sec:hdg-method} we introduce its HDG
discretization. Next, in \Cref{s:well-posedness}, we present the
uniform well-posedness analysis of the discretization in weighted
inner-products and norms. In \Cref{sec:preconditioning-framework} we
present the preconditioning framework and analyze the preconditioners
for the reduced problem. We provide numerical examples in
\Cref{s:numex} and end with conclusions in \Cref{s:conclusions}.

\section{The coupled Stokes--Darcy problem}
\label{sec:stokes-darcy-sys}

To describe the model let us consider a bounded polygonal domain
$\Omega\subset \mathbb{R}^{\dim}$, where $\dim = 2,3$. The boundary of
the domain is denoted by $\partial \Omega$. Due the coupled nature of
the system, we partition $\Omega$ into disjoint regions $\Omega^s$ and
$\Omega^d$ such that $\Omega = \Omega^s \cup \Omega^d$. The external
boundary of the \emph{Stokes} region $\Omega^s$ is denoted by
$\Gamma^s := \partial\Omega \cap \partial\Omega^s$ while the external
boundary of the \emph{Darcy} region $\Omega^d$ is denoted by
$\Gamma^d := \partial\Omega \cap \partial\Omega^d$. The interface
between the two regions is denoted by
$\Gamma^I := \partial\Omega^s \cap \partial\Omega^d$.

The coupled Stokes--Darcy system is given by:
\begin{subequations}
  \label{eq:stokes-darcy}
  \begin{align}
  \label{eq:stokes-darcy-a}
    -\nabla \cdot 2 \mu \varepsilon(u) + \nabla p
    &= f^s && \text{ in } \Omega^s,
    \\
    \label{eq:stokes-darcy-b}
    \mu \kappa^{-1} u + \nabla p 
    &= 0 && \text{ in } \Omega^d,
    \\
    \label{eq:stokes-darcy-c}
    - \nabla \cdot u
    &= \chi^d f^d && \text{ in } \Omega,
    \\
    \label{eq:stokes-darcy-d}
    u
    &= 0 && \text{ on } \Gamma^s,
    \\
    \label{eq:stokes-darcy-e}
    u \cdot n
    & = 0 && \text{ on } \Gamma^d,
  \end{align}
\end{subequations}
where $\varepsilon(u) := (\nabla u + (\nabla u)^T)/2$ is the symmetric
gradient and $\chi^d$ is the characteristic function with value equal
to $1$ in the Darcy region and $0$ otherwise. Furthermore,
$f^s \in [L^2(\Omega^s)]^{\dim}$ and $f^d \in L^2(\Omega^d)$ are
forcing terms, $\mu > 0$ is the kinematic viscosity, and $\kappa > 0$
is the permeability constant. For ease of notation we will denote the
restriction of the velocity $u$ and pressure $p$ to $\Omega^j$ by
$u^j$ and $p^j$, respectively, for $j = s,d$.

On the interface $\Gamma^I$, we impose the following transfer
conditions:
\begin{subequations}
  \label{eq:interface-cond}
  \begin{align}
  \label{eq:interface-cond-a}
    u^s \cdot n 
    &= u^d\cdot n && \text{ on } \Gamma^I,
    \\
    \label{eq:interface-cond-b}
    p^s - 2 \mu \varepsilon(u^s) n \cdot n
    &= p^d && \text{ on } \Gamma^I,
    \\
    \label{eq:interface-cond-c}
    - 2 \mu (\varepsilon(u^s) n)^t
    &= \alpha \mu \kappa^{-1/2} (u^s)^t && \text{ on } \Gamma^I,
  \end{align}
\end{subequations}
where $n$ is the unit normal vector on the interface $\Gamma^I$
(outward with respect to $\Omega^s$), $(v)^t := v - (v \cdot n) n$
denotes the tangential component of a vector $v$, and
\cref{eq:interface-cond-c} is the Beavers--Joseph--Saffman law
\cite{beavers1967boundary,saffman1971boundary} where $\alpha > 0$ is
an experimentally determined dimensionless constant. Following
\cite[Section 3]{boon2020parameter}, we assume that there exists a
constant $c_{\alpha} > 0$, independent of $\mu$, $\kappa$, and
$\alpha$, such that
\begin{equation}
  \label{eq:alpha-kappa-rel}
  \alpha \le c_{\alpha} \kappa^{1/2}.
\end{equation}

\section{The hybridizable discontinuous Galerkin discretization}
\label{sec:hdg-method}

Let $j = s,d$, denote by $\mathcal{T}_h^j := \{K\}$ a family of
simplicial triangulations of the domain $\Omega^j$, and let
$\mathcal{T}_h := \mathcal{T}_h^s \cup \mathcal{T}_h^d$. We define
$h := \max_{K \in \mathcal{T}_h} h_K$, where $h_K$ is the diameter of
an element $K \in \mathcal{T}_h$. The outward unit normal vector on
$\partial K$ is denoted by $n^j$ if $K \in \mathcal{T}_h^j$, for
$j = s, d$. If no confusion can occur, we will drop the superscript
and write $n$. We assume that the mesh consists of shape-regular
elements.  Additionally, we denote the
set of all faces by $\mathcal{F}_h$ and the union of all faces by
$\Gamma_0$. The set of all faces corresponding to the interface
$\Gamma^I$ is denoted by $\mathcal{F}_h^I$, the set of all faces lying
in $\Bar{\Omega}^j$, is denoted by $\mathcal{F}_h^j$, and the union of
all faces in $\Bar{\Omega}^j$ is denoted by $\Gamma^j_0$ ($j = s,d$).

We introduce the following finite element function spaces:
\begin{align*}
    V_h &:= \{ v_h \in [L^2(\Omega)]^{\dim}:\ 
    v_h \in [\mathbb{P}_k (K)]^{\dim}\ \forall K \in \mathcal{T}_h\}, \\
    V_h^j &:= \{ v_h \in [L^2(\Omega^j)]^{\dim}:\ 
    v_h \in [\mathbb{P}_k (K)]^{\dim}\ \forall K \in \mathcal{T}_h^j\}, \\
    Q_h &:= \{ q_h \in L_0^2(\Omega):\ 
    q_h \in \mathbb{P}_{k-1} (K)\ \forall K \in \mathcal{T}_h\}, \\
    Q_h^j &:= \{ q_h \in L^2(\Omega^j):\ 
    q_h \in \mathbb{P}_{k-1} (K)\ \forall K \in \mathcal{T}_h^j\},\ j = \{s,d\},\\
    \bar{V}_h &:= \{ \bar{v}_h \in [L^2(\Gamma_0^s)]^{\dim}:\ 
    \bar{v}_h \in [\mathbb{P}_k (F)]^{\dim}\ \forall F \in \mathcal{F}_h^s,\ \bar{v}_h = 0 \text{ on } \Gamma^s\},\\
    \bar{Q}_h^j &:= \{ \bar{q}_h^j \in L^2(\Gamma_0^j):\ 
    q_h^j \in \mathbb{P}_{k}(F)\ \forall F \in \mathcal{F}_h^j\},\ j = \{s,d\},
\end{align*}
where $\mathbb{P}_r(K)$ and $\mathbb{P}_r(F)$ denote the sets of
polynomials of degree at most $r$ on a cell $K$ and face $F$,
respectively, and
$L_0^2(\Omega) := \{q \in L^2(\Omega) : \int_{\Omega} q \dif x = 0
\}$. For ease of notation, we define the following function space
pairs: $\boldsymbol{V}_h := V_h \times \bar{V}_h$,
$\boldsymbol{Q}_h := Q_h \times \bar{Q}_h^s \times \bar{Q}_h^d $,
$\boldsymbol{V}_h^s := V_h^s \times \bar{V}_h$, and
$\boldsymbol{Q}_h^j := Q_h^j \times \bar{Q}_h^j$, for $j =
s,d$. Function pairs are also denoted in bold face, for example,
$\boldsymbol{v}_h := (v_h, \bar{v}_h) \in \boldsymbol{V}_h$,
$\boldsymbol{q}_h := (q_h, \bar{q}_h^s, \bar{q}_h^d) \in
\boldsymbol{Q}_h$,
$\boldsymbol{v}_h^s := (v_h^s, \bar{v}_h) \in \boldsymbol{V}_h^s$, and
$\boldsymbol{q}_h^j = (q_h, \bar{q}_h^j) \in
\boldsymbol{Q}_h^j$. Furthermore, we set $X_h := V_h \times Q_h$,
$\bar{X}_h := \bar{V}_h \times \bar{Q}_h^s \times \bar{Q}_h^d$, and
$\boldsymbol{X}_h := X_h \times \bar{X}_h$.

Let $D \subset \mathbb{R}^{\dim}$ be a Lipschitz domain. We denote the
$L^2$-inner product and norm on $D$ by $(\cdot, \cdot)_D$ and
$\norm[0]{\cdot}_D$, respectively. Additionally, we denote the
$L^2$-inner product on a submanifold $S$ of codimension 1 by
$\langle \cdot, \cdot \rangle_S$. We then define the following inner
products:
$(\phi, \psi)_{\mathcal{T}_h} := \sum_{K \in \mathcal{T}_h}
(\phi,\psi)_K$ and
$\langle \phi, \psi\rangle_{\partial\mathcal{T}_h} := \sum_{K \in
  \mathcal{T}_h} \langle \phi, \psi \rangle_{\partial K}$ if $\phi$
and $\psi$ are scalars and with similar definitions holding for the
vector case. The norms induced by these inner products are denoted by
$\norm[0]{\cdot}_{\mathcal{T}_h}$ and
$\norm[0]{\cdot}_{\partial \mathcal{T}_h}$.

The HDG method for the coupled Stokes--Darcy system is given by
\cite{cesmelioglu2020embedded}: Find
$\boldsymbol{x}_h := (\boldsymbol{u}_h, \boldsymbol{p}_h) \in
\boldsymbol{X}_h$ such that
\begin{equation}
  \label{eq:hdg-discretization}
  a_h(\boldsymbol{x}_h, \boldsymbol{y}_h)
  = (f^s, v_h)_{\mathcal{T}_h^s} + (f^d, q_h)_{\mathcal{T}_h^d}
  \quad \forall \boldsymbol{y}_h := (\boldsymbol{v}_h, \boldsymbol{q}_h) \in \boldsymbol{X}_h,
\end{equation}
where
\begin{align*}
  a_h(\boldsymbol{x}_h, \boldsymbol{y}_h)
  :=& d_h(\boldsymbol{u}_h, \boldsymbol{v}_h)
      + b_h(\boldsymbol{p}_h, \boldsymbol{v}_h)
      + b_h(\boldsymbol{q}_h, \boldsymbol{u}_h),
  \\
  d_h(\boldsymbol{u}_h, \boldsymbol{v}_h)
  :=& d_h^s(\boldsymbol{u}_h, \boldsymbol{v}_h)
      + d_h^d(u_h, v_h)
      + d_h^I(\bar{u}_h, \bar{v}_h),
  \\
  d_h^s(\boldsymbol{u}_h, \boldsymbol{v}_h)
  :=& 2 \mu (\varepsilon(u_h), \varepsilon(v_h))_{\mathcal{T}_h^s}
      + 2 \mu \eta \langle h_K^{-1} (u_h - \bar{u}_h), v_h - \bar{v}_h\rangle_{\partial \mathcal{T}_h^s}
  \\
    & - 2 \mu \langle \varepsilon(u_h) n^s, v_h - \bar{v}_h\rangle_{\partial \mathcal{T}_h^s}
      - 2 \mu \langle \varepsilon(v_h) n^s, u_h - \bar{u}_h\rangle_{\partial\mathcal{T}_h^s},
  \\
  d_h^d(u_h, v_h) 
  :=& \mu \kappa^{-1} (u_h, v_h)_{\mathcal{T}_h^d},
  \\
  d_h^I(\bar{u}_h, \bar{v}_h)
  :=& \alpha \mu \kappa^{-1/2} \langle \bar{u}_h^t, \bar{v}_h^t\rangle_{\Gamma^I},
  \\
  b_h(\boldsymbol{p}_h, \boldsymbol{v}_h)
  :=& \sum_{j = s, d} \del[1]{ b_h^j(\boldsymbol{p}_h^j, v_h) 
      + b_h^{I, j}(\bar{p}_h^j, \bar{v}_h) },
  \\
  b_h^j(\boldsymbol{p}_h^j, v_h) 
  :=& - (p_h, \nabla \cdot v_h)_{\mathcal{T}_h^j}
      + \langle \bar{p}_h^j, v_h \cdot n^j\rangle_{\partial \mathcal{T}_h^j},
  \\
  b_h^{I, j}(\bar{p}_h^j, \bar{v}_h) 
  :=& - \langle \bar{p}_h^j, \bar{v}_h \cdot n^j \rangle_{\Gamma^I},
\end{align*}
where
$\langle h_K^{-1} (u_h - \bar{u}_h), v_h - \bar{v}_h\rangle_{\partial
  \mathcal{T}_h^s}:= \sum_{K \in \mathcal{T}_h^s} \langle h_K^{-1}
(u_h - \bar{u}_h), v_h - \bar{v}_h\rangle_{\partial K}$ and $\eta > 1$
is a sufficiently large penalty parameter. In the numerical
experiments in \cref{s:numex} we choose $\eta = 4 k^2$.

\section{Uniform well-posedness}
\label{s:well-posedness}

To prove uniform well-posedness of the HDG method introduced in
\cref{eq:hdg-discretization} we require weighted inner products. These
are defined in \cref{ss:weightedinprod}. The norms induced by these
inner products are then used to prove uniform boundedness and uniform
stability of the HDG bilinear form $a_h(\cdot, \cdot)$ in
\cref{ss:boundedness,ss:stability}, respectively, and as a
consequence, uniform well-posedness. In what follows, we will refer to
a positive constant that is independent of $\mu$, $\kappa$, and $h$ as
a uniform constant.

\subsection{Weighted inner-products}
\label{ss:weightedinprod}

We consider the following weighted inner products for
$\boldsymbol{u}_h, \boldsymbol{v}_h \in \boldsymbol{V}_h$ and
$\boldsymbol{p}_h, \boldsymbol{q}_h \in \boldsymbol{Q}_h$:
\begin{subequations}
  \label{eq:inner-product}
  \begin{align}
    (\boldsymbol{u}_h, \boldsymbol{v}_h)_v 
    &:= (\boldsymbol{u}_h^s, \boldsymbol{v}_h^s)_{v,s} 
      +(u_h^d, v_h^d)_{v,d}
      + \alpha \mu \kappa^{-1/2} \langle \bar{u}_h^t, \bar{v}_h^t\rangle_{\Gamma^I},
    \\
    (\boldsymbol{q}_h, \boldsymbol{p}_h)_q
    &:= (\boldsymbol{p}_h^s, \boldsymbol{q}_h^s)_{q,s}
      + (\boldsymbol{p}_h^d, \boldsymbol{q}_h^d)_{q,d}
      + (\alpha \mu \kappa^{-1/2})^{-1} \langle \bar{p}_h^d, \bar{q}_h^d \rangle_{\Gamma^I},    
  \end{align}
  where, for
  $\boldsymbol{u}_h^s, \boldsymbol{v}_h^s \in \boldsymbol{V}_h^s$,
  $u_h^d,v_h^d \in V_h^d$,
  $\boldsymbol{p}_h^s, \boldsymbol{q}_h^s \in \boldsymbol{Q}_h^s$, and
  $\boldsymbol{p}_h^d, \boldsymbol{q}_h^d \in \boldsymbol{Q}_h^d$,  
  \begin{align}
    (\boldsymbol{u}_h^s, \boldsymbol{v}_h^s)_{v,s}
    &:= 2 \mu (\varepsilon(u_h^s), \varepsilon(v_h^s))_{\mathcal{T}_h^s}
      + 2 \mu \eta \langle h_K^{-1} (u_h^s - \bar{u}_h^s), v_h^s - \bar{v}_h^s \rangle_{\partial\mathcal{T}_h^s},
    \\
    (u_h^d, v_h^d)_{v,d}
    &:= \mu \kappa^{-1} (u_h^d, v_h^d)_{\mathcal{T}_h^d},
    \\
    (\boldsymbol{p}_h^s, \boldsymbol{q}_h^s)_{q,s}
    &:= (2 \mu)^{-1} (p_h^s, q_h^s)_{\mathcal{T}_h^s}
      + (2 \mu \eta)^{-1} \langle h_K\bar{p}_h^s, \bar{q}_h^s\rangle_{\partial \mathcal{T}_h^s}, 
    \\
    (\boldsymbol{p}_h^d, \boldsymbol{q}_h^d)_{q,d}
    &:= \mu^{-1} \kappa (\nabla p_h^d, \nabla q_h^d)_{\mathcal{T}_h^d}
      + \mu^{-1} \kappa \eta \langle h_K^{-1} (p_h^d - \bar{p}_h^d), q_h^d - \bar{q}_h^d \rangle_{\partial 
      \mathcal{T}_h^d}.    
  \end{align}
\end{subequations}
The norms induced by the above inner products are denoted by
$\tnorm{\boldsymbol{v}_h}_{v}$, $\tnorm{\boldsymbol{q}_h}_{q}$,
$\tnorm{\boldsymbol{v}_h^s}_{v,s}$, $\tnorm{v_h^d}_{v,d}$,
$\tnorm{\boldsymbol{q}_h^s}_{q,s}$, and
$\tnorm{\boldsymbol{q}_h^d}_{q,d}$, respectively, while
\begin{equation}
  \label{eq:normXhSD}
  \tnorm{\boldsymbol{y}_h}_{\boldsymbol{X}_h}^2 :=
  \tnorm{\boldsymbol{v}_h}_v^2 + \tnorm{\boldsymbol{q}_h}_q^2
  \quad \forall \boldsymbol{y}_h := (v_h,q_h,\bar{v}_h,\bar{q}_h^s,\bar{q}_h^d) \in \boldsymbol{X}_h.
\end{equation}
Additionally, we introduce the following norm
\begin{equation}
  \label{eq:facet-norm-v}
  \tnorm{\bar{v}_h}_{v,h}^2 := \norm[0]{h_K^{-1/2}(\bar{v}_h - m_K(\bar{v}_h))}_{\partial \mathcal{T}_h}^2 
  \quad \forall \bar{v}_h \in \Bar{V}_h,
\end{equation}
where
$m_K(\bar{v}_h) := |\partial K|^{-1} \int_{\partial K} \bar{v}_h \dif
s$. In addition, we define $M_K(v_h) := |K|^{-1} \int_{K} v_h \dif x$.

\subsection{Boundedness}
\label{ss:boundedness}

\begin{lemma}[Boundedness]
  \label{lem:boundedness}
  There exists a uniform constant $c_b$ such that
  \begin{equation}
    \label{eq:boundedness}
    |a_h(\boldsymbol{x}_h, \boldsymbol{y}_h)|
    \leq c_b \tnorm{\boldsymbol{x}_h}_{\boldsymbol{X}_h} \tnorm{\boldsymbol{y}_h}_{\boldsymbol{X}_h}
    \quad \forall \boldsymbol{x}_h, \boldsymbol{y}_h \in \boldsymbol{X}_h.
  \end{equation}
\end{lemma}
\begin{proof}
  Combining the estimate of \cite[Lemma 4.3]{rhebergen2017analysis}
  with the definition of the norms and H\"older's inequality for sums,
  we deduce that there exists a uniform constant $c$ such that
  \begin{equation}
    \label{eq:boundedness-aux1}
    d_h(\boldsymbol{u}_h, \boldsymbol{v}_h)
    \leq c \tnorm{\boldsymbol{u}_h}_v \tnorm{\boldsymbol{v}_h}_v 
    \quad \forall \boldsymbol{u}_h, \boldsymbol{v}_h \in \boldsymbol{V}_h.
  \end{equation}
  Next, following the steps from the proof of \cite[Lemma
  3]{cesmelioglu2023hybridizable2}, but using the weighted norms, we
  find that there exists a uniform constant $c$ such that
  \begin{equation}
    \label{eq:boundedness-aux2}
    b_h^s(\boldsymbol{q}_h^s, v_h)
    + b_h^{I,s}(\bar{q}_h^s, \bar{v}_h)
    \leq c \tnorm{\boldsymbol{q}_h}_q \tnorm{\boldsymbol{v}_h}_v
    \quad \forall (\boldsymbol{v}_h, \boldsymbol{q}_h)
    \in \boldsymbol{V}_h \times \boldsymbol{Q}_h.
  \end{equation}
  Using integration by parts, the Cauchy--Schwarz inequality with
  weighted norms, and the discrete trace inequality, we deduce that
  there exists a uniform constant $c$ such that (see also
  \cite[Appendix A.1]{henriquez2025parameter}),
  \begin{equation}
    \label{eq:boundedness-aux3}
    b_h^d(\boldsymbol{q}_h^d, {v}_h^d)
    \leq c \tnorm{\boldsymbol{q}_h}_{q} \tnorm{\boldsymbol{v}_h}_v
    \quad \forall (\boldsymbol{v}_h, \boldsymbol{q}_h) 
    \in \boldsymbol{V}_h \times \boldsymbol{Q}_h.
  \end{equation}
  Finally, note that by the Cauchy--Schwarz inequality,
  \cref{eq:interface-norm-bound-2,eq:alpha-kappa-rel} we find:
  \begin{equation}
    \label{eq:boundedness-aux4}
    b_h^{I,d}(\bar{q}_h^d, \bar{v}_h)
    \leq (\alpha \kappa^{-1/2} \mu)^{-1/2} \norm[0]{\bar{q}_h^d}_{\Gamma^I}
    (\alpha \kappa^{-1/2} \mu)^{1/2} \norm[0]{\bar{v}_h}_{\Gamma^I}
    \leq c_{\alpha}^{1/2}c_\Gamma \tnorm{\boldsymbol{q}_h}_q \tnorm{\boldsymbol{v}_h}_v.
  \end{equation}
  Therefore, combining
  \cref{eq:boundedness-aux1,eq:boundedness-aux2,eq:boundedness-aux3,eq:boundedness-aux4}
  we deduce \cref{eq:boundedness}. \qed
\end{proof}
\subsection{Stability and well-posedness}
\label{ss:stability}

Before proceeding with the proof of uniform stability, we recall some
known results. By \cite[Lemma 2]{cesmelioglu2020embedded}, we have
that there exists a uniform constant $c_1$, such that
\begin{equation}
  \label{eq:coercivity-dh}
  d_h(\boldsymbol{v}_h, \boldsymbol{v}_h) 
  \geq c_1 \big( \tnorm{\boldsymbol{v}_h}_{v,s}^2
  + \tnorm{v_h}_{v,d}^2 
  + \alpha \mu \kappa^{-1/2} \norm[0]{\bar{v}_h^t}_{\Gamma^I}^2 \big)
  \quad \forall \boldsymbol{v}_h \in \boldsymbol{V}_h.
\end{equation}
It is shown in \cite[Lemma 1]{rhebergen2018preconditioning} that there
exists a uniform constant $c_2$
\begin{equation}
  \label{eq:inf-sup-bhs}
  \sup_{\boldsymbol{0} \neq \boldsymbol{v}_h \in \boldsymbol{V}_h^s}
  \dfrac{b_h^s(\boldsymbol{q}_h^s, v_h) 
  + b_h^{I,s}(\bar{q}_h^s, \bar{v}_h)}{\tnorm{\boldsymbol{v}_h}_{v,s}}
  \geq c_2 \tnorm{\boldsymbol{q}_h^s}_{q,s}
  \quad \forall \boldsymbol{q}_h^s \in \boldsymbol{Q}_h^s.
\end{equation}
We furthermore prove in Appendix \ref{ap:proofinfsup} that there
exists a uniform constant $c_3$, such that
\begin{equation}
  \label{eq:inf-sup-bhd}
  \sup_{0 \neq v_h \in V_h^d}
  \dfrac{b_h^d(\boldsymbol{q}_h^d, v_h)}{\tnorm{v_h}_{v,d}}
  \geq c_3 \tnorm{\boldsymbol{q}_h^d}_{q,d} 
  \quad \forall \boldsymbol{q}_h^d \in \boldsymbol{Q}_h^d.
\end{equation}
The inf-sup conditions \cref{eq:inf-sup-bhs,eq:inf-sup-bhd} are used
to prove the following global inf-sup condition.

\begin{lemma}[Stability]
  \label{lem:inf-sup-bh}
  There exists a uniform constant $c_s$ such that
  \begin{equation}
    \label{eq:inf-sup-bh}
    \sup_{0 \neq \boldsymbol{v}_h \in \boldsymbol{V}_h}
    \dfrac{b_h(\boldsymbol{q}_h, \boldsymbol{v}_h)}{\tnorm{\boldsymbol{v}_h}_v}
    \geq c_s 
    \tnorm{\boldsymbol{q}_h}_{q}^2 
    \quad \forall \boldsymbol{q}_h \in \boldsymbol{Q}_h.
  \end{equation}
\end{lemma}
\begin{proof}
  By \cref{eq:inf-sup-bhs} we have that given
  $\boldsymbol{q}_h^s \in \boldsymbol{Q}_h^s$ there exists
  $\boldsymbol{w}_h^s \in \boldsymbol{V}_h^s$ such that
  \begin{equation}
    \label{eq:inf-sup-bh-aux1}
    b_h^s(\boldsymbol{q}_h^s, w_h^s) 
    + b_h^{I,s}(\bar{q}_h^s, \bar{w}_h^s)
    = \tnorm{\boldsymbol{q}_h^s}_{q,s}^2
    \quad\text{ and }\quad
    \tnorm{\boldsymbol{w}_h^s}_{v,s} \leq c_2^{-1} \tnorm{\boldsymbol{q}_h^s}_{q,s},
  \end{equation}
  and by \cref{eq:inf-sup-bhd} we have that given
  $\boldsymbol{q}_h^d \in \boldsymbol{Q}_h^d$ there exists
  $w_h^d \in V_h^d$ such that
  \begin{equation}
    \label{eq:inf-sup-bh-aux2}
    b_h^d(\boldsymbol{q}_h^d, w_h^d) 
    = \tnorm{\boldsymbol{q}_h^d}_{q,d}^2
    \quad\text{ and }\quad
    \tnorm{w_h^d}_{v,d} \leq c_3^{-1} \tnorm{\boldsymbol{q}_h^d}_{q,d}. 
  \end{equation}
  Let $\delta$ be a uniform constant that we will choose later. Let
  $\boldsymbol{w}_h = \boldsymbol{w}_h^s$ on $\mathcal{T}_h^s$ and
  $\boldsymbol{w}_h = \delta (w_h^d,0)$ on $\mathcal{T}_h^d$. Then,
  using \cref{eq:inf-sup-bh-aux1,eq:inf-sup-bh-aux2},
  \begin{equation}
    \label{eq:inf-sup-bh-aux3}
    \begin{split}
      \sup_{0 \neq \boldsymbol{v}_h \in \boldsymbol{V}_h}
      \frac{b_h(\boldsymbol{q}_h, \boldsymbol{v}_h)}{\tnorm{\boldsymbol{v}_h}_v}
      & \ge
        \frac{b_h(\boldsymbol{q}_h, \boldsymbol{w}_h)}{\tnorm{\boldsymbol{w}_h}_v}
      \\
      & \ge
        \dfrac{\tnorm{\boldsymbol{q}_h^s}_{q,s}^2 
        + \delta \tnorm{\boldsymbol{q}_h^d}_{q,d}^2
        - \langle\bar{q}_h^d, \bar{w}_h^s \cdot n^d\rangle_{\Gamma^I} }{(c_2^{-2} \tnorm{\boldsymbol{q}_h^{s}}_{q,s}^2
        + c_3^{-2} \delta^2 \tnorm{\boldsymbol{q}_h^d}_{q,d}^2
        + \alpha \kappa^{-1/2} \mu \norm[0]{\bar{w}_h^s}_{\Gamma^I}^2)^{1/2}}.
        \end{split}
    \end{equation}
    By the Cauchy--Schwarz inequality,
    \cref{eq:alpha-kappa-rel,eq:inf-sup-bh-aux1,eq:interface-norm-bound-2},
    and Young's inequality, we find:
    \begin{equation}
      \label{eq:inf-sup-bh-aux4}
      \begin{split}
        \langle \bar{q}_h^d, \bar{w}_h^s \cdot n^d \rangle_{\Gamma^I}
        & \leq 
          (\alpha^{-1} \kappa^{1/2} \mu^{-1})^{1/2} \norm[0]{\bar{q}_h^d}_{\Gamma^I}
          (\alpha \kappa^{-1/2} \mu)^{1/2} \norm[0]{\bar{w}_h^s}_{\Gamma^I}
        \\
        & \leq 
          \tfrac{1}{2} c_2^{-2} c_{\Gamma}^2 c_{\alpha} \alpha^{-1} \kappa^{1/2} \mu^{-1} \norm[0]{\bar{q}_h^d}_{\Gamma^I}^2
          + \tfrac{1}{4} \tnorm{\boldsymbol{q}_h^s}_{q,s}^2.        
      \end{split}
    \end{equation}
    Using \cref{eq:interface-norm-bound-4,eq:alpha-kappa-rel} we
    furthermore find that
    \begin{equation}
      \label{eq:inf-sup-bh-aux5}
      \tfrac{1}{2} c_2^{-2} c_{\Gamma}^2 c_{\alpha} \alpha^{-1} \kappa^{1/2} \mu^{-1} \norm[0]{\bar{q}_h^d}_{\Gamma^{I}}^2
      \le \tfrac{1}{2} c_2^{-2} c_{\Gamma}^4 c_{\alpha}^2 \alpha^{-2} \tnorm{\boldsymbol{q}_h^d}_{q,d}^2,
    \end{equation}
    and using
    \cref{eq:alpha-kappa-rel,eq:interface-norm-bound-2,eq:inf-sup-bh-aux1}
    we find that
    \begin{equation}
      \label{eq:inf-sup-bh-aux6}
      \alpha \kappa^{-1/2} \mu \norm[0]{\bar{w}_h^s}_{\Gamma^I}^2
      \le \tfrac{1}{2} c_{\Gamma}^2 c_{\alpha} c_2^{-2}\tnorm{\boldsymbol{q}_h^s}_{q,s}^2.
    \end{equation}
    Therefore, choosing
    $\delta = c_{\alpha}^2 c_2^{-2} c_{\Gamma}^4 \alpha^{-2}$, and
    combining
    \cref{eq:inf-sup-bh-aux3,eq:inf-sup-bh-aux4,eq:inf-sup-bh-aux5,eq:inf-sup-bh-aux6}, we find
    \begin{align*}
        \sup_{0 \neq \boldsymbol{v}_h \in \boldsymbol{V}_h}
        \frac{b_h(\boldsymbol{q}_h, \boldsymbol{v}_h)}{\tnorm{\boldsymbol{v}_h}_v}
        & \ge \dfrac{\frac{3}{4} \tnorm{\boldsymbol{q}_h^s}_{q,s}^2
        + \delta \tnorm{\boldsymbol{q}_h^d }_{q,d}^2
        - \frac{1}{2} c_2^{-2} c_{\Gamma}^2 c_{\alpha} \alpha^{-1} \kappa^{1/2} \mu^{-1} \norm[0]{\bar{q}_h^d }_{\Gamma^I}^2 }{
        (( c_2^{-2} + \frac{1}{2} c_{\Gamma}^2 c_{\alpha} c_{2}^{-2} )\tnorm{\boldsymbol{q}_h^s}_{q,s}^2
        + c_3^{-2} \delta^2 \tnorm{\boldsymbol{q}_h^d}_{q,d}^2 )
        } \\
        & \ge \dfrac{\min( \frac{3}{4}, \frac{1}{2} c_2^{-2} c_{\Gamma}^4 c_{\alpha}^2 \alpha^{-2}) }{\max(c_2^{-2} + \frac{1}{2} c_{\Gamma}^2 c_{\alpha} c_2^{-2}, c_3^{-2} c_2^{-4} c_{\Gamma}^8 c_{\alpha}^4 \alpha^{-4} ) } (\tnorm{\boldsymbol{q}_h^s}_{q,s}^2 + \tnorm{\boldsymbol{q}_h^d}_{q,d}^2)^{1/2}.
    \end{align*}
    Using \cref{eq:inf-sup-bh-aux5}, we note that
    \begin{align*}
      \tnorm{\boldsymbol{q}_h^s}_{q,s}^2
      + \tnorm{\boldsymbol{q}_h^d}_{q,d}^2
      &= \tnorm{\boldsymbol{q}_h^s}_{q,s}^2
        + \tfrac{1}{2}\tnorm{\boldsymbol{q}_h^d}_{q,d}^2
        + \tfrac{1}{2}\tnorm{\boldsymbol{q}_h^d}_{q,d}^2
      \\
      &\ge \tnorm{\boldsymbol{q}_h^s}_{q,s}^2
        + \tfrac{1}{2}\tnorm{\boldsymbol{q}_h^d}_{q,d}^2
        + \tfrac{1}{2} c_{\alpha}^{-1} c_{\Gamma}^{-2} \alpha^2 (\alpha \kappa^{-1/2} \mu)^{-1} \norm[0]{\bar{q}_h^d}_{\Gamma^{I}}^2
      \\
      &\ge \tfrac{1}{2} \min(1, c_{\alpha}^{-1} c_{\Gamma}^{-2} \alpha^2) \tnorm{\boldsymbol{q}_h}_q^2,
    \end{align*}
    and so \cref{eq:inf-sup-bh} follows. \qed
\end{proof}

We conclude this section with the following theorem on uniform
well-posedness. 

\begin{theorem}[Uniform well-posedness]
  \label{thm:uniformwellposed}
  The HDG discretization \cref{eq:hdg-discretization} is uniformly
  well-posed in the norms induced by the inner products defined in
  \cref{eq:inner-product}.
\end{theorem}
\begin{proof}
  The Brezzi conditions (see \cite[Chapter 4]{boffi2013mixed}) are 
  satisfied by invoking \cref{lem:boundedness,lem:inf-sup-bh}. 
  Consequently, the statement holds. \qed
\end{proof}

\section{Preconditioning}
\label{sec:preconditioning-framework}

One of the main features of hybridizable discretizations is the ease
at which it is possible to eliminate local degrees of freedom. In this
section we will determine a preconditioner for the reduced problem
obtained after eliminating these local degrees of freedom. In
\cref{ss:mardal-winther} we first summarize the theoretical framework
introduced in \cite{mardal2011preconditioning} to construct
parameter-robust preconditioners for symmetric discretizations. Its
extension in \cite{henriquez2025parameter} to hybridizable
discretizations to construct parameter-robust preconditioners for the
reduced problem is summarized in \cref{ss:Static-condensation}. In the
remainder of this section we then use this framework to construct a
parameter-robust preconditioner for the reduced problem obtained from
the HDG discretization of the Stokes--Darcy problem
\cref{eq:hdg-discretization}.

\subsection{Operator preconditioning framework}
\label{ss:mardal-winther}

We begin by considering an arbitrary finite element space
$\boldsymbol{X}_h$ depending on a mesh size $h$. We denote by
$(\cdot, \cdot)_{\boldsymbol{X}_h}$ an inner product on
$\boldsymbol{X}_h$ and by $\norm[0]{\cdot}_{\boldsymbol{X}_h}$ the
norm induced by this inner product. Additionally, we denote the
duality pairing of $\boldsymbol{X}_h^*$ and $\boldsymbol{X}_h$ by
$\langle\cdot, \cdot\rangle_{\boldsymbol{X}_h^*, \boldsymbol{X}_h}$,
where $\boldsymbol{X}_h^*$ is the dual space of
$\boldsymbol{X}_h$. 

Consider the problem of finding
$\boldsymbol{x}_h \in \boldsymbol{X}_h$ such that
\begin{equation}
  \label{eq:general-probl}
  a_h(\boldsymbol{x}_h, \boldsymbol{y}_h) 
  = \langle\boldsymbol{f}_h, \boldsymbol{y}_h\rangle_{\boldsymbol{X}_h^*, \boldsymbol{X}_h}
  \quad \forall \boldsymbol{y}_h \in \boldsymbol{X}_h,
\end{equation}
where the bilinear form $a_h(\cdot,\cdot)$ defined on
$\boldsymbol{X}_h\times\boldsymbol{X}_h$ is symmetric and
$h$-dependent, and $\boldsymbol{f}_h \in \boldsymbol{X}_h^*$. This
problem is well-posed if it satisfies
\begin{subequations}
  \label{eq:gnral-boundedness}
  \begin{align}
    \label{eq:gnral-boundedness-1}
    a_h(\boldsymbol{x}_h, \boldsymbol{y}_h)
    & \le c_b \norm[0]{\boldsymbol{x}_h}_{\boldsymbol{X}_h}\norm[0]{\boldsymbol{y}_h}_{\boldsymbol{X}_h}
      \quad \forall \boldsymbol{x}_h, \boldsymbol{y}_h \in \boldsymbol{X}_h,
    \\
    \label{eq:gnral-boundedness-2}
    c_i &\le \inf_{\boldsymbol{x}_h \in \boldsymbol{X}_h} \sup_{\boldsymbol{y}_h \in \boldsymbol{X}_h} 
          \dfrac{a_h(\boldsymbol{x}_h, \boldsymbol{y}_h)}{\norm[0]{\boldsymbol{x}_h}_{\boldsymbol{X}_h} \norm[0]{\boldsymbol{y}_h}_{\boldsymbol{X}_h} },
  \end{align}
\end{subequations}
where $c_b$ and $c_i$ are the positive boundedness and stability
constants, respectively.

If we define $A : \boldsymbol{X}_h \to \boldsymbol{X}_h^*$ such that
$\langle A \boldsymbol{x}_h,
\boldsymbol{y}_h\rangle_{\boldsymbol{X}_h^*, \boldsymbol{X}_h} =
a_h(\boldsymbol{x}_h, \boldsymbol{y}_h)$ for all
$\boldsymbol{x}_h, \boldsymbol{y}_h \in \boldsymbol{X}_h$, we can
rewrite \cref{eq:general-probl} as: Find
$\boldsymbol{x}_h \in \boldsymbol{X}_h$ such that
\begin{equation}
  \label{eq:general-probl-2}
  A \boldsymbol{x}_h = \boldsymbol{f}_h \text{ in }\boldsymbol{X}_h^*.
\end{equation}
A preconditioner $P^{-1}: \boldsymbol{X}_h^* \to \boldsymbol{X}_h$ for
problem \cref{eq:general-probl-2} can be defined as
\begin{equation}
  \label{eq:gnral-precond}
  (P^{-1} \boldsymbol{f}_h, \boldsymbol{y}_h)_{\boldsymbol{X}_h} 
  = \langle \boldsymbol{f}_h, \boldsymbol{y}_h\rangle_{\boldsymbol{X}_h^*, \boldsymbol{X}_h}
  \quad\boldsymbol{y}_h \in \boldsymbol{X}_h,\, \boldsymbol{f}_h \in \boldsymbol{X}_h^*.
\end{equation}
Let $\mathcal{L}(\boldsymbol{X}_h, \boldsymbol{X}_h)$ be the set of
bounded linear operators mapping $\boldsymbol{X}_h$ to
$\boldsymbol{X}_h$ and let
\begin{subequations}
  \label{eq:precond-norm}
  \begin{align}
    \label{eq:precond-norm-1}
    \norm[0]{P^{-1}A}_{\mathcal{L}(\boldsymbol{X}_h, \boldsymbol{X}_h)}
    & = \sup_{\boldsymbol{x}_h, \boldsymbol{y}_h} \dfrac{(P^{-1}A \boldsymbol{x}_h, \boldsymbol{y}_h)_{\boldsymbol{X}_h}}{\norm[0]{\boldsymbol{x}_h}_{\boldsymbol{X}_h} \norm[0]{\boldsymbol{y}_h}_{\boldsymbol{X}_h}},
    \\
    \label{eq:precond-norm-2}
    \norm[0]{(P^{-1}A)^{-1}}_{\mathcal{L}(\boldsymbol{X}_h, \boldsymbol{X}_h)}^{-1}
    & = \inf_{\boldsymbol{x}_h \in \boldsymbol{X}_h} 
      \sup_{\boldsymbol{y}_h \in \boldsymbol{X}_h}
      \dfrac{(P^{-1}A \boldsymbol{x}_h, \boldsymbol{y}_h)_{\boldsymbol{X}_h}}{\norm[0]{\boldsymbol{x}_h}_{\boldsymbol{X}_h} \norm[0]{\boldsymbol{y}_h}_{\boldsymbol{X}_h}}.
  \end{align}
\end{subequations}
An estimate for the condition number $\text{cond}(P^{-1}A)$ is given
by
\begin{equation*}
  \text{cond}(P^{-1}A) 
  = \norm[0]{P^{-1}A}_{\mathcal{L}(\boldsymbol{X}_h, \boldsymbol{X}_h)}
  \norm[0]{(P^{-1}A)^{-1}}_{\mathcal{L}(\boldsymbol{X}_h, \boldsymbol{X}_h)}^{-1}
  \le c_b/c_i.
\end{equation*}
If the constants $c_b$ and $c_i$ are independent of model parameters
and mesh-size $h$, then $P^{-1}$ is a parameter-robust preconditioner.

\subsection{Hybridizable preconditioning framework}
\label{ss:Static-condensation}

Assume $\boldsymbol{X}_h := X_h \times \bar{X}_h$ and
$\boldsymbol{X}_h \ni \boldsymbol{x}_h = (x_h, \bar{x}_h)$ in which
$x_h$ represent the local degrees of freedom. We then write
\cref{eq:general-probl-2} as
\begin{equation}
  \label{eq:abstract-Ax=F-2}
  \underbrace{\begin{bmatrix}
    A_{11} & A_{21}^T \\
    A_{21} & A_{22}    
  \end{bmatrix}}_{=: A}
\begin{bmatrix}
  x_h \\ \bar{x}_h
\end{bmatrix}
=
\begin{bmatrix}
  f_h \\ \bar{f}_h
\end{bmatrix},
\end{equation}
in which $A_{11} : X_h \to X_h^*$, $A_{21} : X_h \to \bar{X}_h^*$, and
$A_{22} : \bar{X}_h \to \bar{X}_h^*$ are the operators that form
$A: \boldsymbol{X}_h \to \boldsymbol{X}_h^*$, and
$\boldsymbol{f}_h = (f_h, \bar{f}_h) \in X_h^*\times
\bar{X}_h^*$. Eliminating the local degrees of freedom $x_h$, the
reduced problem is given by
\begin{equation}
  \label{eq:gen-cond-syst}
  \underbrace{(A_{22} - A_{21} A_{11}^{-1} A_{21}^T)}_{=:S_A} \bar{x}_h
  = \bar{f}_h - A_{21} A_{11}^{-1} f_h.
\end{equation}
We remark that $S_{A}$ is the Schur complement of the matrix $A$ of
the original system. The local degrees of freedom can be recovered by
solving
\begin{equation*}
  x_h = A_{11}^{-1} f_h - A_{11}^{-1} A_{21}^{T} \bar{x}_h.
\end{equation*}
Assume that preconditioner $P$, defined by \cref{eq:gnral-precond},
has the same block structure as $A$ in \cref{eq:abstract-Ax=F-2}:
\begin{equation}
  \label{eq:gnral-precond-matrix}
  P =
  \begin{bmatrix}
    P_{11} & P_{21}^T \\
    P_{21} & P_{22}    
  \end{bmatrix},
\end{equation}
where $P_{11}: X_h \to X_h^*$, $P_{21}: X_h \to \bar{X}_h^*$,
$P_{22}: \bar{X}_h \to \bar{X}_h^*$, and with Schur complement
\begin{equation}
  \label{eq:gnral-precond-schur}
  S_P:= P_{22} - P_{21}P_{11}^{-1}P_{21}^T.
\end{equation}
Assume the Schur complement is a symmetric and positive operator in
the sense that
$\langle S_P \bar{x}_h, \bar{x}_h\rangle_{\bar{X}_h^*, \bar{X}_h} > 0$
for all $\bar{x}_h \in \bar{X}_h \backslash \{0\}$. The operator
$S_P:\bar{X}_h \to \bar{X}_h^*$ therefore defines an inner product on
the space $\bar{X}_h$ as follows:
\begin{equation}
  \label{eq:SPinnerproduct}
  \langle S_P \bar{x}_h, \bar{y}_h\rangle_{\bar{X}_h^*, \bar{X}_h}
  = (\bar{x}_h, \bar{y}_h)_{\bar{X}_h}
  \quad \forall \bar{x}_h, \bar{y}_h \in \bar{X}_h.
\end{equation}
The norm induced by $(\cdot, \cdot)_{\bar{X}_h}$ is denoted by
$\norm[0]{\cdot}_{\bar{X}_h}$. The idea behind the hybridizable
preconditioning framework is to use $S_P$, the Schur complement of the
block-matrix representation of the bilinear form associated with the
inner product $(\cdot, \cdot)_{\boldsymbol{X}_h}$, i.e., of $P$, as a
preconditioner for $S_A$, the reduced system obtained after applying
static condensation to \cref{eq:abstract-Ax=F-2}. Assuming that
$A_{11}$ is invertible and that $P_{11}$ is a positive operator, we
proved in \cite[Theorem 2.3]{henriquez2025parameter} that $S_P$ is a
parameter-robust preconditioner for $S_A$ if and only if there exists
a uniform constant $c_l$ such that
\begin{equation}
  \label{eq:robust-schur-precon-1}
  \norm[0]{(-A_{11}^{-1}A_{21}^T \bar{x}_h, \bar{x}_h)}_{\boldsymbol{X}_h}
  \leq c_l \norm[0]{\bar{x}_h}_{\bar{X}_h}
  \quad \forall \bar{x}_h \in \bar{X}_h.
\end{equation}

\subsection{Preconditioners}
\label{ss:preconds}

We start this section by defining the inner products
\begin{subequations}
  \label{eq:inner-prod-redef}
  \begin{align}
    \label{eq:inner-prod-redef-a}
    (\boldsymbol{u}_h^s, \boldsymbol{v}_h^s)_{v,s'}
    & := (\boldsymbol{u}_h^s, \boldsymbol{v}_h^s)_{v,s} 
      + \alpha \mu \kappa^{-1/2} \langle \bar{u}_h^t, \bar{v}_h^t\rangle_{\Gamma^I}
    && \forall \boldsymbol{u}_h^s, \boldsymbol{v}_h^s \in \boldsymbol{V}_h^s,
    \\
    \label{eq:inner-prod-redef-b}
    (\boldsymbol{p}_h^d, \boldsymbol{q}_h^d)_{q,d'}
    & := (\boldsymbol{p}_h^d, \boldsymbol{q}_h^d)_{q,d}
      + \alpha^{-1} \mu^{-1} \kappa^{1/2}  \langle \bar{p}_h^d, \bar{q}_h^d \rangle_{\Gamma^I}
    && \forall \boldsymbol{p}_h^d, \boldsymbol{q}_h^d \in \boldsymbol{Q}_h^d,
  \end{align}
\end{subequations}
and the operators
$P^{u_s} : \boldsymbol{V}_h^s \to \boldsymbol{V}_h^{s,*}$,
$P^{u_d} : V_h^d \to V_h^{d,*}$,
$P^{p_s} : \boldsymbol{Q}_h^s \to \boldsymbol{Q}_h^{s,*}$, and
$P^{p_d} : \boldsymbol{Q}_h^d \to \boldsymbol{Q}_h^{d,*}$ such that
\begin{equation}
  \label{eq:preconditionerPSD}
  \begin{split}
    \langle P \boldsymbol{x}_h, \boldsymbol{y}_h \rangle_{\boldsymbol{X}_h^*, \boldsymbol{X}_h}
    =& (\boldsymbol{u}_h^s, \boldsymbol{v}_h^s)_{v,s'} 
       + (u_h^d, v_h^d)_{v,d}
       + (\boldsymbol{p}_h^s, \boldsymbol{q}_h^s)_{q,s}
       + (\boldsymbol{p}_h^d, \boldsymbol{q}_h^d)_{q,d'}
    \\
    =& \langle P^{u_s} \boldsymbol{u}_h^s, \boldsymbol{v}_h^s \rangle_{\boldsymbol{V}_h^{s,*}, \boldsymbol{V}_h^s}
       + \langle P^{u_d} u_h^d, v_h^d \rangle_{V_h^{d,*}, V_h^d}
    \\
     & + \langle P^{p_s} \boldsymbol{p}_h^s, \boldsymbol{q}_h^s \rangle_{\boldsymbol{Q}_h^{s,*}, \boldsymbol{Q}_h^{s}}
       + \langle P^{p_d} \boldsymbol{p}_h^d, \boldsymbol{q}_h^d \rangle_{\boldsymbol{Q}_h^{d,*}, \boldsymbol{Q}_h^{d}}.    
  \end{split}
\end{equation}
We have the following block form representation of $P$:
\begin{equation*}
  P
  =
  \begin{bmatrix}
    P^{u_s} & 0 & 0 & 0
    \\
    0 & P^{u_d} & 0 & 0
    \\
    0 & 0 & P^{p_s} & 0
    \\
    0 & 0 & 0 & P^{p_d}
  \end{bmatrix}
  =
  \begin{bmatrix}
    P_{11}^{u_s} & (P_{21}^{u_s})^T & 0 & 0 & 0 & 0 & 0
    \\
    P_{21}^{u_s} & P_{22}^{u_s} & 0 & 0 & 0 & 0 & 0
    \\
    0 & 0 & P^{u_d} & 0 & 0 & 0 & 0
    \\
    0 & 0 & 0 & P_{11}^{p_s} & 0 & 0 & 0
    \\
    0 & 0 & 0 & 0 & P_{22}^{p_s} & 0 & 0
    \\
    0 & 0 & 0 & 0 & 0 & P_{11}^{p_d} & (P_{21}^{p_d})^T
    \\
    0 & 0 & 0 & 0 & 0 & P_{21}^{p_s} & P_{22}^{p_d}
  \end{bmatrix},
\end{equation*}
where $P_{11}^{u_s} : V_h^s \to V_h^{s,*} $,
$P_{21}^{u_s} : V_h^s \to \bar{V}_h^{*} $,
$P_{22}^{u_s} : \bar{V}_h \to \bar{V}_h^{*}$,
$P_{11}^{p_s} : Q_h^s \to Q_h^{s,*}$,
$P_{22}^{p_s} : \bar{Q}_h^s \to \bar{Q}_h^{s,*}$,
$P_{11}^{p_d} : Q_h^d \to Q_h^{d,*}$, and
$P_{21}^{p_d} : Q_h^d \to \bar{Q}_h^{d,*}$,
$P_{22}^{p_d} : \bar{Q}_h^d \to \bar{Q}_h^{d,*}$. By the uniform
well-posedness analysis of \cref{s:well-posedness} and the discussion
in \cref{ss:mardal-winther} we note that the operator $P$ defines a
parameter robust preconditioner for the HDG discretization of the
coupled Stokes--Darcy problem \cref{eq:hdg-discretization}
\emph{before} static condensation. We furthermore claim that the Schur
complement of $P$,
$S_P : (\bar{V}_h \times \bar{Q}_h^s \times \bar{Q}_h^d) \to
(\bar{V}_h^* \times \bar{Q}_h^{s,*} \times \bar{Q}_h^{d,*})$ is a
parameter-robust preconditioner for the reduced form of
\cref{eq:hdg-discretization}. To prove this, we need to verify
\cref{eq:robust-schur-precon-1}. We do this in
\cref{ss:condensed-preconditioner}.

To end this section, let us note that $S_P$ can be written as
\begin{equation}
  \label{eq:Sp}
  S_P
  =
  \begin{bmatrix}
    S_{P^{u_s}} & 0 & 0
    \\
    0 & S_{P^{p_s}} & 0
    \\
    0 & 0 & S_{P^{p_d}} 
  \end{bmatrix},
\end{equation}
where
$S_{P^{u_s}} = P_{22}^{u_s} -
P_{21}^{u_s}(P_{11}^{u_s})^{-1}(P_{21}^{u_s})^T$,
$S_{P^{p_s}} = P_{22}^{p_s} $, and
$S_{P^{p_d}} = P_{22}^{p_d} -
P_{21}^{p_d}(P_{11}^{p_d})^{-1}(P_{21}^{p_d})^T$. Furthermore, we
introduce the following inner products on the facet spaces:
\begin{equation}
  \label{eq:face-inner-prod}
  (\bar{u}_h, \bar{v}_h)_{\bar{v},s} 
  := \langle S_{P^{u_s}} \bar{u}_h, \bar{v}_h\rangle_{\bar{V}_h^*, \bar{V}_h},
  \qquad
  (\bar{p}_h^d, \bar{q}_h^d)_{\bar{q}, d} 
  := \langle S_{P^{p_d}} \bar{q}_h^d, \bar{q}_h^d\rangle_{\bar{Q}_h^{d,*}, \bar{Q}_h^{d}},
\end{equation}
and their induced norms $\tnorm{\cdot}_{\bar{v}, s}$ and
$\tnorm{\cdot}_{\bar{q}, d}$.

\subsection{Auxiliary problems}
\label{sec:aux-problem}

We introduce two auxiliary problems that will be used in the analysis
of preconditioner $S_P$.

\subsubsection{Auxiliary problem for the velocity field in $\Omega^s$} 
\label{sec:aux-problem-velocity}

We consider the following vector diffusion problem:
\begin{equation*}
  - \nabla \cdot 2 \mu \varepsilon(u) = f \text{ in } \Omega^s,
  \quad 2 \mu \varepsilon (u) n = 0 \text{ on } \Gamma^I, 
  \quad u = 0 \text{ on } \Gamma^s.
\end{equation*}
For the discretization of this problem, we use the HDG discretization
presented in \cite{wells2011analysis}: Given $f \in [L^2(\Omega)]^d$
find $\boldsymbol{u}_h \in \boldsymbol{V}_h^s$ such that
\begin{equation}
  \label{eq:aux-u-problem}
  d_h^s(\boldsymbol{u}_h, \boldsymbol{v}_h) = (f, v_h)_{\mathcal{T}_h}
  \quad \forall \boldsymbol{v}_h \in \boldsymbol{V}_h^s.
\end{equation}
We have the following result.

\begin{lemma}
  \label{lem:coer-bound-dhs}
  There exist uniform constants $c_1^s, c_2^s$ such that for all
  $\boldsymbol{v}_h \in \boldsymbol{V}_h^s$,
  \begin{equation}
    \label{eq:coer-bound-dhs}
    c_1^s \big(
    \tnorm{\boldsymbol{v}_h}_{v,s}^2
    + \alpha \mu \kappa^{-1/2} \norm[0]{\bar{v}_h^t}_{\Gamma^I}^2
    \big)
    \leq 
    d_h^s(\boldsymbol{v}_h, \boldsymbol{v}_h)
    \leq 
    c_2^s \big(
    \tnorm{\boldsymbol{v}_h}_{v,s}^2
    + \alpha \mu \kappa^{-1/2} \norm[0]{\bar{v}_h^t}_{\Gamma^I}^2
    \big).
  \end{equation}
\end{lemma}
\begin{proof}
  By the coercivity result in \cite[Lemma 4.2]{rhebergen2017analysis},
  with uniform constant $c$, and using
  \cref{eq:alpha-kappa-rel,eq:interface-norm-bound-2}, we deduce
    \begin{align*}
      d_h^s(\boldsymbol{v}_h, \boldsymbol{v}_h)
      \geq \dfrac{c}{2} \tnorm{\boldsymbol{v}_h}_{v,s}^2
      + c c_{\Gamma}^{-2} \mu \norm[0]{\bar{v}_h}_{\Gamma^I}^2 
      \geq \min(c/2, c c_{\Gamma}^{-2} c_{\alpha}^{-1})
      \big( \tnorm{\boldsymbol{v}_h}_{v,s}^2
      + \alpha \mu \kappa^{-1/2} \norm[0]{\bar{v}_h}_{\Gamma^I}^2
      \big),
    \end{align*}
    for all $\boldsymbol{v}_h \in \boldsymbol{V}_h^s$. The upper bound
    of \cref{eq:coer-bound-dhs} follows immediately by
    \cite[Lemma 4.3]{rhebergen2017analysis}. \qed
\end{proof}

Next, given $\bar{m}_h \in \bar{V}_h$ and
$z \in [L^2(\Omega^s)]^{\dim}$, we define the function
$\tilde{u}_h^L(\bar{m}_h, z) \in V_h^s$ such that, when restricted to
$K \in \mathcal{T}_h^s$, it satisfies the local problem
\begin{equation}
  \label{eq:local-ds=fs}
  d_h^{s,K}(\tilde{u}_h^{L}, v_h) = f_h^{s,K}(v_h)
  \quad \forall v_h \in V(K) := [\mathbb{P}_k(K)]^d,
\end{equation}
where
\begin{align*}
  d_h^{s,K}(u_h, v_h)
  := d_h^s((u_h,0),(v_h,0))|_K
  \quad \text{and} \quad
  f_h^{s,K}(v_h)
  := (z,v_h)_K - d_h^s((0,\bar{m}_h),(v_h,0))|_K,
\end{align*}
in which $d_h^s(\cdot, \cdot)|_K$ is the restriction of
$d_h^s(\cdot, \cdot)$ to a cell $K$. Given
$f \in [L^2(\Omega^s)]^{\dim}$, let us now define
$\tilde{u}_h^f := u_h^L(0, f)$ and
$\tilde l_u(\bar{v}_h) := u_h^L(\bar{v}_h, 0)$ for all
$\bar{v}_h \in \bar{V}_h$. Eliminating $u_h$ from
\cref{eq:aux-u-problem} we find the following reduced problem for
$\bar{u}_h$:
\begin{equation}
  \label{eq:condensed-auxproblem-u}
  d_h^s((\tilde{l}_u(\bar{u}_h), \bar{u}_h), (\tilde{l}_u(\bar{v}_h), \bar{v}_h))
  = (f, \tilde{l}_u(\bar{v}_h))_{\mathcal{T}_h^s}
  \quad \forall \bar{v}_h \in \bar{V}_h,
\end{equation}
and remark that $(u_h, \bar{u}_h)$, where
$u_h = \tilde{u}_h^f + \tilde{l}_u(\bar{u}_h)$, solves
\cref{eq:aux-u-problem}. (The proof is similar to that of \cite[Lemma
4]{rhebergen2018preconditioning} and so it is omitted.)

We now define the operator
$D^s : \boldsymbol{V}_h \to \boldsymbol{V}_h^*$ by
$\langle D^s \boldsymbol{u}_h, \boldsymbol{v}_h
\rangle_{\boldsymbol{V}_h^*, \boldsymbol{V}_h} =
d_h^s(\boldsymbol{u}_h, \boldsymbol{v}_h)$ for all
$\boldsymbol{u}_h, \boldsymbol{v}_h \in \boldsymbol{V}_h$. This
operator has the block structure
\begin{equation*}
  D^s
  =
  \begin{bmatrix}
    D^s_{11} & (D^{s}_{21})^T \\
    D^s_{21} & D^s_{22}
  \end{bmatrix},
\end{equation*}
where $D^s_{11} : V_h \to V_h^*$, $D^s_{21} : V_h \to \bar{V}_h^*$,
and $D^s_{22} : \bar{V}_h \to \bar{V}_h^*$, and Schur complement
operator $S_{D^s} : \bar{V}_h \to \bar{V}_h^*$ where
$S_{D^s} := D^s_{22} - D^s_{21} D^s_{11} (D^{s}_{21})^T$. 

\subsubsection{Auxiliary problem for the pressure field in $\Omega^d$}
\label{sec:aux-problem-pressure}

Consider the following diffusion problem:
\begin{equation*}
  - \nabla \cdot (\mu^{-1} \kappa \nabla p) = g \text{ in } \Omega^d,
  \quad \nabla p \cdot n = 0 \text{ on } \partial \Omega^d, 
  \quad \int_{\Omega^d} p \dif x = 0,
\end{equation*}
where $g$ is a given source term. An interior penalty HDG
discretization of this problem is given by \cite{wells2011analysis}:
Given $g \in L^2(\Omega^d)$, find
$\boldsymbol{p}_h \in \boldsymbol{Q}_0^d := (Q_h^d \cap
L_0^2(\Omega^d))\times \bar{Q}_h^d$, such that
\begin{equation}
  \label{eq:aux-p-problem}
  \Tilde{d}_h(\boldsymbol{p}_h, \boldsymbol{q}_h)
  = (g, q_h)_{\mathcal{T}_h^d} 
  \quad \forall \boldsymbol{q}_h \in \boldsymbol{Q}_0^d,
\end{equation}
where
\begin{equation}
  \label{eq:tilde-a-def}
  \begin{split}
    \Tilde{d}_h(\boldsymbol{p}_h, \boldsymbol{q}_h)
    :=& \mu^{-1} \kappa (\nabla p_h, \nabla q_h)_{\mathcal{T}_h^d}
        + \mu^{-1} \kappa \eta \langle h_K^{-1} (p_h - \bar{p}_h), q_h - \bar{q}_h \rangle_{\partial \mathcal{T}_h^d}
    \\
      & - \mu^{-1} \kappa \langle \nabla p_h \cdot n, q_h - \bar{q}_h \rangle_{\partial \mathcal{T}_h^d}
        - \mu^{-1} \kappa \langle \nabla q_h \cdot n, p_h - \bar{p}_h \rangle_{\partial \mathcal{T}_h^d}.    
  \end{split}
\end{equation}
We have the following result.

\begin{lemma}
  There exist uniform constants $\Tilde{c}_1'$, $\Tilde{c}_2'$,
  $\tilde{c}_1$, and $\tilde{c}_2$ such that for all
  $ \boldsymbol{q}_h \in \boldsymbol{Q}_0^d$:
  \begin{subequations}
    \begin{align}
      \label{eq:coer-bound-tilde-a}
      \Tilde{c}_1' \tnorm{\boldsymbol{q}_h}_{q,d}^2
      &\leq \Tilde{d}_h(\boldsymbol{q}_h, \boldsymbol{q}_h)
      \leq \Tilde{c}_2' \tnorm{\boldsymbol{q}_h}_{q,d}^2,
      \\
      \label{eq:coer-bound-tilde-a2}
      \tilde{c}_1 
      \big( \tnorm{\boldsymbol{q}_h}_{q,d}^2
      + \alpha^{-1} \kappa^{1/2} \mu^{-1} \norm[0]{\bar{q}_h^d}_{\Gamma^I}^2 \big)
      &\leq 
        \tilde{d}_h(\boldsymbol{q}_h, \boldsymbol{q}_h)
        \leq \tilde{c}_2 
        \big( \tnorm{\boldsymbol{q}_h}_{q,d}^2
        + \alpha^{-1} \kappa^{1/2} \mu^{-1} \norm[0]{\bar{q}_h^d}_{\Gamma^I}^2 \big).
    \end{align}
  \end{subequations}
\end{lemma}
\begin{proof}
  \Cref{eq:coer-bound-tilde-a} follows from \cite[Lemmas 5.2 and
  5.3]{wells2011analysis}. To show \cref{eq:coer-bound-tilde-a2} we
  first note that by
  \cref{eq:alpha-kappa-rel,eq:interface-norm-bound-4} we have
  \begin{align*}
    \Tilde{c}_1' \tnorm{\boldsymbol{q}_h}_{q,d}^2
    & \ge \tfrac{1}{2} \Tilde{c}_1' \tnorm{\boldsymbol{q}_h}_{q,d}^2
      + \tfrac{1}{2} \Tilde{c}_1' c_{\Gamma}^{-2} \alpha c_{\alpha}^{-1} \kappa^{1/2} \mu^{-1} \norm[0]{\bar{q}_h}_{\Gamma^I}^2
    \\
    &\ge c \min(1, c_{\alpha}^{-1} \alpha^2) \big( 
      \tnorm{\boldsymbol{q}_h}_{q,d}^2
      + \alpha^{-1} \kappa^{1/2} \mu^{-1} \norm[0]{\bar{q}_h}_{\Gamma^I}^2
      \big),
  \end{align*}
  for all $\boldsymbol{q}_h \in \boldsymbol{Q}_0^d$. The result
  follows after combining with \cref{eq:coer-bound-tilde-a}. \qed
\end{proof}

Next, similar to what we did in \cref{sec:aux-problem-velocity}, we
define the function
$\Tilde{p}_h^L(\bar{t}_h^d, s) \in Q_h^d \cap L_0^2(\Omega^d)$, for
given $\bar{t}_h^d \in \bar{Q}_h^d$ and $s \in L^2(\Omega^d)$, to be
such that when it is restricted to $K \in \mathcal{T}_h^d$ it
satisfies
\begin{equation}
  \label{eq:local-tilde-ad=f}
  \tilde{d}_h^{K}(\tilde{p}_h^L, q_h) 
  = \tilde{g}_h^K(q_h) \quad \forall q_h \in Q(K) \cap L_0^2(\Omega^d),
\end{equation}
where
\begin{align*}
  \tilde{d}_h^{K}(\tilde{p}_h^L, q_h) 
  := \tilde{d}_h((p_h,0),(q_h,0)),
  \quad
  \tilde{g}_h^K(q_h)
  := (s, q_h)_K - \tilde{d}_h((0,\bar{t}_h^d),(q_h,0)).
\end{align*}
Given $g \in L^2(\Omega^d)$, we define
$\tilde{p}_h^{g} := \tilde{p}_h(0,g)$ and
$\tilde{l}_p(\bar{q}_h^d) := \tilde{p}_h^L(\bar{q}_h^d, 0)$ for all
$\bar{q}_h^d \in \bar{Q}_h^d$. Eliminating $p_h$ from
\cref{eq:aux-p-problem} we find the following reduced problem for
$\bar{p}_h^d \in \bar{Q}_h^d$:
\begin{equation}
  \label{eq:condensed-auxproblem}
  \tilde{d}_h((\tilde{l}_p(\bar{p}_h^d), \bar{p}_h^d), (\tilde{l}_p(\bar{q}_h^d), \bar{q}_h^d))
  = (g, \tilde{l}_p(\bar{q}_h^d))_{\mathcal{T}_h^d}
  \quad \forall \bar{q}_h^d \in \bar{Q}_h^d,
\end{equation}
and remark that $(p_h, \bar{p}_h^d)$, where
$p_h = \tilde{p}_h^{g} + \tilde{l}_p(\bar{q}_h^d)$, solves
\cref{eq:aux-p-problem}. (Again, the proof is similar to that of
\cite[Lemma 4]{rhebergen2018preconditioning} and so it is omitted.)

\subsection{Parameter-robustness of the condensed preconditioner}
\label{ss:condensed-preconditioner}

For analysis purposes, we require the variational form of the HDG
discretization after static condensation (see
\cref{eq:condensed-formulation}). To define this variational
formulation, we require local solvers. Let
$t_h := (u_h, p_h) \in X_h$,
$\bar{t}_h := (\bar{m}_h, \bar{r}_h^s, \bar{r}_h^d)$,
$y_h := (v_h, q_h) \in X_h$, $s^s \in [L^2(\Omega^s)]^{\dim}$, and
$s^d \in L^2(\Omega^d)$. We define the following local bilinear and
linear forms for each cell $K \in \mathcal{T}_h^s$,
\begin{align*}
  a_K^s(t_h, y_h)
  &:= d_h^s((u_h, 0), (v_h, 0))|_{K}
    + b_h^s((p_h,0), v_h)|_{K}
    + b_h^s((q_h,0), u_h)|_{K},
  \\
  L^s_K(y_h)
  &:= (s^s , v_h)_{K}
    - d_h^s((0, \bar{m}_h), (v_h, 0))|_{K}
    - b_h^s((0, \bar{r}_h^s), v_h)|_{K}
\end{align*}
and for each cell $K \in \mathcal{T}_h^d$ we define
\begin{align*}
  a_K^d(t_h, y_h)
  &:= d_h^d(u_h, v_h)|_{K}
    + b_h^d((p_h,0), v_h)|_{K}
    + b_h^d((q_h,0), u_h)|_{K},
  \\
  L^d_K(y_h)
  &:= (s^d, q_h)_{K}
    - b_h^d((0, \bar{r}_h^d), v_h)|_{K}.
\end{align*}
Furthermore, for given $\bar{t}_h \in \bar{X}_h$,
$s^s \in [L^2(\Omega^s)]^{\dim}$, and $s^d \in L^2(\Omega^d)$ we define
the functions $u_h^L(\bar{t}_h, s^s, s^d) \in V_h$ and
$p_h^L(\bar{t}_h, s^s, s^d) \in Q_h$ to be such that when restricted
to a cell $K \in \mathcal{T}_h$ they satisfy the local problems
\begin{equation}
  \label{eq:local-problem}
  a_K^j(t_h^L, y_h) = L_K^j(y_h)
  \quad \forall y_h \in V(K) \times Q(K)
  \quad \forall K \in \mathcal{T}_h^j, 
  \quad j \in \{s,d\},
\end{equation}
where $t_h^L := (u_h^L, p_h^L)$, $V(K) := [\mathbb{P}_k(K)]^{\dim}$,
and $Q(K) := \mathbb{P}_{k-1}(K)$.

Given $f^s \in [L^2(\Omega^s)]^{\dim}$ and $f^d \in L^2(\Omega^d)$ we
define $u_h^f := u_h^L(0, f^s, f^d)$ and
$p_h^f := p_h^L(0, f^s, f^d)$. Furthermore, let
$\bar{x}_h := (\bar{u}_h, \bar{p}_h^s, \bar{p}_h^d) \in \bar{X}_h$,
$\bar{y}_h := (\bar{v}_h, \bar{q}_h^s, \bar{q}_h^d) \in \bar{X}_h$,
$l_u(\bar{y}_h) := u_h^L(\bar{y}_h, 0, 0)$, and
$l_{p}(\bar{y}_h) := p_h^L(\bar{y}_h, 0, 0)$. Then, eliminating $u_h$
and $p_h$ from \cref{eq:hdg-discretization}, we find the following
reduced problem for $\bar{x}_h \in \bar{X}_h$:
\begin{equation}
  \label{eq:condensed-formulation}
  \bar{a}_h(\bar{x}_h, \bar{y}_h)
  = (f^s, l_u(\bar{y}_h))_{\mathcal{T}_h^s}
  + (f^d, l_p(\bar{y}_h))_{\mathcal{T}_h^d}
  \quad \forall \bar{y}_h := (\bar{v}_h,\bar{q}_h^s,\bar{q}_h^d) \in \bar{X}_h,
\end{equation}
where 
\begin{align*}
  \bar{a}_h(\bar{x}_h, \bar{y}_h)
  =& d_h((l_u(\bar{x}_h), \bar{u}_h), (l_u(\bar{y}_h), \bar{v}_h))
     + \sum_{j = s, d} \big(b_h^j((l_p(\bar{x}_h), \bar{p}_h^j), l_u(\bar{y}_h))
     + b_h^{I, j}(\bar{p}_h^j, \bar{v}_h) \big)
  \\
   & + \sum_{j = s, d} \big(b_h^j((l_p(\bar{y}_h), \bar{q}_h^j), l_u(\bar{x}_h))
     + b_h^{I, j}(\bar{q}_h^j, \bar{u}_h) \big).
\end{align*}
It can be shown that
$(u_h, \bar{u}_h, p_h, \bar{p}_h^s, \bar{p}_h^d)$, where
$u_h = u_h^f + l_u(\bar{x}_h)$ and $p_h = p_h^f + l_p(\bar{x}_h)$,
solves \cref{eq:hdg-discretization}. A proof of this result uses
arguments similar to those of the proof of \cite[Lemma
4]{rhebergen2018preconditioning} and are therefore omitted. To lay the
link with \cref{ss:Static-condensation}, $A$ in
\cref{eq:abstract-Ax=F-2} is the matrix associated with
$a_h(\cdot, \cdot)$ in \cref{eq:hdg-discretization} while $S_A$ in
\cref{eq:gen-cond-syst} is the matrix associated with
$\bar{a}_h(\cdot, \cdot)$ in \cref{eq:condensed-formulation}.

The remainder of this section is devoted to proving
\cref{eq:robust-schur-precon-1}. To this end, we first establish a few
estimates that will then be used to prove our main result,
\Cref{thm:robustnesstheorem}.
\begin{lemma}
  Given $(\bar{v}_h, \bar{q}_h^s, \bar{q}_h^d) \in \bar{X}_h$, let
  $l_u(\bar{v}_h, \bar{q}_h^s, \bar{q}_h^d)$ and
  $l_p(\bar{v}_h, \bar{q}_h^s, \bar{q}_h^d)$ be as used to define
  \cref{eq:condensed-formulation}. There exist uniform constants
  $c,\hat{c}$ such that
  \begin{subequations}
    \label{eq:estimate-norm-qsvs}
    \begin{align}
      \label{eq:estimate-norm-vs}
      \tnorm{(l_u(\bar{v}_h, \bar{q}_h^s, \bar{q}_h^d), \bar{v}_h)}_{v,s}^2
      \leq& c \big( 2 \mu \eta \tnorm{\bar{v}_h}_{v, h}^2
            + (2 \mu)^{-1} \eta^{-1} \norm[0]{h_K^{1/2} \bar{q}_h^s}_{\partial \mathcal{T}_h^s}^2 \big),
      \\        
      \label{eq:estimate-norm-qs}
      \tnorm{(l_p(\bar{v}_h, \bar{q}_h^s, \bar{q}_h^d), \bar{q}_h^s)}_{q,s}^2
      \leq& \hat{c} \eta \big( \tnorm{(l_u(\bar{v}_h, \bar{q}_h^s, \bar{q}_h^d), \bar{v}_h)}_{v,s}^2
            + (2 \mu)^{-1} \eta^{-1} \norm[0]{h_K^{1/2} \bar{q}_h^s}_{\partial \mathcal{T}_h^s}^2 \big).
    \end{align}
  \end{subequations} 
\end{lemma}
\begin{proof}
  The proof for \cref{eq:estimate-norm-vs} is analogous to the one
  presented in \cite[Lemma B.2]{henriquez2025parameter} and so
  omitted.

  We now consider \cref{eq:estimate-norm-qs}. To simplify notation we
  write $l_u := l_u(\bar{v}_h, \bar{q}_h^s, \bar{q}_h^d)$ and
  $l_p := l_p(\bar{v}_h, \bar{q}_h^s, \bar{q}_h^d)$. By
  \cref{eq:inf-sup-bhs}, given
  $(l_p, \bar{q}_h^s) \in \boldsymbol{Q}_h^s$, there exists
  $\boldsymbol{w}_h^s \in \boldsymbol{V}_h$ such that
  \begin{equation*}
    b_h^s((l_p, \bar{q}_h^s), w_h^s) 
    = \tnorm{(l_p, \bar{q}_h^s)}_{q,s}^2
    \quad \text{ and } \quad 
    \tnorm{\boldsymbol{w}_h^s}_{v,s}
    \leq c_2^{-1} \tnorm{(l_p, \bar{q}_h^s)}_{q,s}.
  \end{equation*}
  Choose $v_h = w_h^s - m_K(\bar{w}_h^s)$ and $q_h = 0$ in
  \cref{eq:local-problem} ($j=s$) and add over all cells
  $K \in \mathcal{T}_h^s$ to find:
  \begin{align*}
    & 2 \mu (\varepsilon(l_u), \varepsilon(w_h^s))_{\mathcal{T}_h^s}
      - 2 \mu \langle \varepsilon(l_u) n^s, w_h^s - m_K(\bar{w}_h^s) \rangle_{\partial \mathcal{T}_h^s}
      - 2 \mu \langle l_u , \varepsilon(w_h^s) n^s\rangle_{\partial\mathcal{T}_h^s}
    \\
    & 
      + 2 \mu \eta \langle h_K^{-1} l_u, w_h^s - m_K(\bar{w}_h^s) \rangle_{\partial \mathcal{T}_h^s}
      - (\nabla \cdot w_h^s, l_p)_{\mathcal{T}_h^s}
    \\
    =&- 2 \mu \langle \varepsilon(w_h^s) n^s, \bar{v}_h\rangle_{\partial\mathcal{T}_h^s}
       + 2 \mu \eta \langle h_K^{-1} \bar{v}_h, w_h^s - m_K(\bar{w}_h^s)\rangle_{\partial \mathcal{T}_h^s}
       - \langle \bar{q}_h^s, (w_h^s - m_K(\bar{w}_h^s)) \cdot n^s\rangle_{\partial \mathcal{T}_h^s}.
  \end{align*}
  Following analogous steps to those in the proof of \cite[Lemma
  B.1]{henriquez2025parameter}, the result follows. \qed
\end{proof}

\begin{lemma}
  Given $(\bar{v}_h, \bar{q}_h^s, \bar{q}_h^d) \in \bar{X}_h$, let
  $l_u(\bar{v}_h, \bar{q}_h^s, \bar{q}_h^d)$ and
  $l_p(\bar{v}_h, \bar{q}_h^s, \bar{q}_h^d)$ be as used to define
  \cref{eq:condensed-formulation}. There exists a uniform constant $c$
  such that
  \begin{equation}
    \label{eq:estimate-norm-vd}
    \tnorm{l_u(\bar{v}_h, \bar{q}_h^s, \bar{q}_h^d)}_{v,d}^2
    \leq c \tilde{d}_h((\tilde{l}_p(\bar{q}_h^d), \bar{q}_h^d), (\tilde{l}_p(\bar{q}_h^d), \bar{q}_h^d)).
  \end{equation}
\end{lemma}
 
\begin{proof}
  Let us write $l_u := l_u(\bar{v}_h, \bar{q}_h^s, \bar{q}_h^d)$ and
  $l_p := l_p(\bar{v}_h, \bar{q}_h^s, \bar{q}_h^d)$.  Let
  $K \in \mathcal{T}_h^d$. Choosing $v_h = 0$ in
  \cref{eq:local-problem} ($j=d$), we find that
  $\nabla \cdot l_u = 0$. Choosing now $v_h = l_u$ in
  \cref{eq:local-problem} ($j=d$) we find
  \begin{align*}
    \mu \kappa^{-1} \norm[0]{l_u}_K^2
    & = - \langle \bar{q}_h^d, l_u \cdot n \rangle_{\partial K}
      = \langle \tilde{l}_p(\bar{q}_h^d) - \bar{q}_h^d, l_u \cdot n \rangle_{\partial K}
      - \langle \tilde{l}_p(\bar{q}_h^d), l_u \cdot n\rangle_{\partial K}
    \\
    & = \langle \tilde{l}_p(\bar{q}_h^d) - \bar{q}_h^d, l_u \cdot n \rangle_{\partial K} - (\nabla \tilde{l}_p(\bar{q}_h^d), l_u)_K,
  \end{align*}
  where the last equality is by applying integration by parts. Then,
  summing over all $K \in \mathcal{T}_h^d$, applying the
  Cauchy--Schwarz inequality, a discrete trace inequality, Young's
  inequality, and \cref{eq:coer-bound-tilde-a}, the result
  follows. \qed
\end{proof}

\begin{lemma}
  Given $(\bar{v}_h, \bar{q}_h^s, \bar{q}_h^d) \in \bar{X}_h$, let
  $l_u(\bar{v}_h, \bar{q}_h^s, \bar{q}_h^d)$ and
  $l_p(\bar{v}_h, \bar{q}_h^s, \bar{q}_h^d)$ be as used to define
  \cref{eq:condensed-formulation}. Then
  \begin{equation}
    \label{eq:estimate-norm-qd}
    \tnorm{(l_p(\bar{v}_h, \bar{q}_h^s, \bar{q}_h^d), \bar{q}_h^d)}_{q,d}
    \leq c_3^{-1} \tnorm{l_u(\bar{v}_h, \bar{q}_h^s, \bar{q}_h^d)}_{v,d}.
  \end{equation}
\end{lemma}

\begin{proof}
  Write $l_u := l_u(\bar{v}_h, \bar{q}_h^s, \bar{q}_h^d)$ and
  $l_p := l_p(\bar{v}_h, \bar{q}_h^s, \bar{q}_h^d)$. Thanks to the
  inf-sup condition \cref{eq:inf-sup-bhd}, given
  $\bar{q}_h^d \in \bar{Q}_h^d$ there exists $w_h^d \in V_h$ such that
  \begin{subequations}
    \begin{align}
      \label{eq:bhdinfsupequal}
      b_h^d((l_p(\bar{v}_h, \bar{q}_h^s, \bar{q}_h^d), \bar{q}_h^d), w_h^d)
      & = \tnorm{(l_p(\bar{v}_h, \bar{q}_h^s, \bar{q}_h^d), \bar{q}_h^d)}_{q,d}^2,
      \\
      \label{eq:bhdinfsupvdle}
      \tnorm{w_h^d}_{v,d}
      & \leq c_3^{-1} \tnorm{(l_p(\bar{v}_h, \bar{q}_h^s, \bar{q}_h^d), \bar{q}_h^d)}_{q,d}.      
    \end{align}      
  \end{subequations}
  Choose $v_h = w_h^d$ and $q_h = 0$ in \cref{eq:local-problem}, sum
  over all $K \in \mathcal{T}_h^d$, and reorder terms to find
  \begin{equation*}
    b_h^d((l_p, \bar{q}_h^d), w_h^d)
    = - \mu \kappa^{-1} (l_u, w_h^d)_{\mathcal{T}_h^d}.
  \end{equation*}
  Using \cref{eq:bhdinfsupequal}, and applying the Cauchy--Schwarz
  inequality, we find
  \begin{equation*}
    \tnorm{(l_p, \bar{q}_h^d)}_{q,d}^2
    = - \mu \kappa^{-1} (l_u, w_h^d)_{\mathcal{T}_h^d}
    \le \tnorm{l_u}_{v,d} \tnorm{w_h^d}_{v,d}.
  \end{equation*}
  The result follows after using \cref{eq:bhdinfsupvdle}. \qed
\end{proof}

We now present the main result, i.e., we show that
\cref{eq:robust-schur-precon-1} holds for the norms
$\norm[0]{\cdot}_{\boldsymbol{X}_h}$, defined in \cref{eq:normXhSD},
and $\norm[0]{\cdot}_{\bar{X}_h}$, the norm induced by the inner
product $(\cdot, \cdot)_{\bar{X}_h}$ defined in
\cref{eq:SPinnerproduct} with $S_P$ defined in \cref{eq:Sp}.

\begin{theorem}
  \label{thm:robustnesstheorem}
  Given
  $\bar{y}_h := (\bar{v}_h, \bar{q}_h^s, \bar{q}_h^d) \in \bar{X}_h$,
  let $l_u(\bar{y}_h)$ and $l_p(\bar{y}_h)$ be as used to define
  \cref{eq:condensed-formulation}. There exists a uniform constant $c$
  such that
  \begin{equation*}
    \tnorm{(l_u(\bar{y}_h), \bar{v}_h, l_p(\bar{y}_h), \bar{q}_h^s, \bar{q}_h^d)}_{\boldsymbol{X}_h}^2 
    \leq c \eta \big( \eta \tnorm{\bar{v}_h}_{\bar{v},s}^2
    + (2 \mu)^{-1} \eta^{-1} \norm[0]{h_K^{1/2} \bar{q}_h^s}_{\partial \mathcal{T}_h^s}^2
    + \tnorm{\bar{q}_h^d}_{\bar{q}, d}^2 \big).
  \end{equation*}
\end{theorem}
\begin{proof}
  To simplify notation we write $l_u := l_u(\bar{y}_h)$ and
  $l_p := l_p(\bar{y}_h)$. Combining
  \cref{eq:estimate-norm-qsvs,eq:estimate-norm-vd,eq:estimate-norm-qd},
  we find
  \begin{align*}
    \tnorm{(l_u, \bar{v}_h, l_p, \bar{q}_h^s, \bar{q}_h^d)}_{\boldsymbol{X}_h}^2
    \leq & c \eta \big( 2 \mu \eta \tnorm{\bar{v}_h}_{v,h}^2
           + (2 \mu)^{-1} \eta^{-1} \norm[0]{h_K^{1/2} \bar{q}_h^s}_{\partial \mathcal{T}_h^s}^2 \big)           
    \\    
         & + c \tilde{d}_h((\tilde{l}_p(\bar{q}_h^d), \bar{q}_h^d), (\tilde{l}_p(\bar{q}_h^d), \bar{q}_h^d))
           + \alpha \mu \kappa^{-1/2} \norm[0]{\bar{v}_h^t}_{\Gamma^I}^2
           + (\alpha \mu \kappa^{-1/2})^{-1} \norm[0]{\bar{q}_h^d}_{\Gamma^I}^2.    
  \end{align*}
  We bound the $\norm[0]{\cdot}_{\Gamma^I}$-terms using
  \cref{eq:coer-bound-dhs,eq:coer-bound-tilde-a2} to find
  \begin{equation}
    \label{eq:assumnption-thm-aux1}
    \begin{split}
      \tnorm{(l_u, \bar{v}_h, l_p, \bar{q}_h^s, \bar{q}_h^d)}_{\boldsymbol{X}_h}^2
      \leq & c \eta \big( 2 \mu \eta \tnorm{\bar{v}_h}_{v,h}^2
             + (2 \mu)^{-1} \eta^{-1} \norm[0]{h_K^{1/2} \bar{q}_h^s}_{\partial \mathcal{T}_h^s}^2 \big)
      \\
           & + c \tilde{d}_h((\tilde{l}_p(\bar{q}_h^d), \bar{q}_h^d), (\tilde{l}_p(\bar{q}_h^d), \bar{q}_h^d))
             + c d_h^s((\tilde{l}_u(\bar{v}_h), \bar{v}_h), (\tilde{l}_v(\bar{v}_h), \bar{v}_h)).
    \end{split}
  \end{equation}
  Let $\tilde{l}_u(\cdot)$ be as used to define
  \cref{eq:condensed-auxproblem-u}. Applying
  \cref{eq:facet-norm-bound} with
  $(\tilde{l}_u(\bar{v}_h), \bar{v}_h)$, in combination with the
  coercivity of $d_h^s$ (cf. \cref{eq:coer-bound-dhs}) we obtain
  \begin{equation*}
    2 \mu \eta \tnorm{\bar{v}_h}_{v,h}^2
    \leq \bar{c}^2 \tnorm{(\tilde{l}_u(\bar{v}_h), \bar{v}_h)}_{v,s}^2
    \leq \bar{c}^2 (c_1^s)^{-1} d_h^s((\tilde{l}_u(\bar{v}_h), \bar{v}_h), (\tilde{l}_u(\bar{v}_h), \bar{v}_h)).
  \end{equation*}
  Noting that
  $d_h^s((\tilde{l}_u(\bar{v}_h), \bar{v}_h), (\tilde{l}_u(\bar{v}_h),
  \bar{v}_h)) = \langle S_{D^s} \bar{v}_h, \bar{v}_h
  \rangle_{\bar{V}_h^*, \bar{V}_h}$ and applying
  \cref{eq:schur-comp-equiv-2} in combination with
  \cref{eq:coer-bound-dhs} we find that
  \begin{equation}
    \label{eq:assumnption-thm-aux2}
    2 \mu \eta \tnorm{\bar{v}_h}_{v, h}^2
    \leq c \eta \tnorm{\bar{v}_h}_{\bar{v},s}^2.
  \end{equation}
  Furthermore, combining \cref{eq:coer-bound-tilde-a2} with
  \cref{eq:schur-comp-equiv-2}, we find
  \begin{equation}
    \label{eq:assumnption-thm-aux3}
    \tilde{d}_h((\tilde{l}_p(\bar{q}_h), \bar{q}_h), (\tilde{l}_p(\bar{q}_h), \bar{q}_h))
    \leq c \tnorm{\bar{q}_h^d}_{\bar{q},d}^2.
  \end{equation}
  The results follows after combining
  \cref{eq:assumnption-thm-aux1,eq:assumnption-thm-aux2,eq:assumnption-thm-aux3}. \qed
\end{proof}

\subsection{An alternative preconditioner}
\label{ss:altprecon}

The preconditioner $P$ is defined in \cref{eq:preconditionerPSD}
by:
\begin{multline}
  \label{eq:Pfirst}
  \langle P \boldsymbol{x}_h, \boldsymbol{y}_h \rangle_{\boldsymbol{X}_h^*, \boldsymbol{X}_h}
  = (\boldsymbol{u}_h^s, \boldsymbol{v}_h^s)_{v,s} 
  + \alpha \mu \kappa^{-1/2} \langle \bar{u}_h^t, \bar{v}_h^t\rangle_{\Gamma^I}
  + (u_h^d, v_h^d)_{v,d}
  \\
  + (\boldsymbol{p}_h^s, \boldsymbol{q}_h^s)_{q,s}
  + (\boldsymbol{p}_h^d, \boldsymbol{q}_h^d)_{q,d}
  + \alpha^{-1} \mu^{-1} \kappa^{1/2}  \langle \bar{p}_h^d, \bar{q}_h^d \rangle_{\Gamma^I}.
\end{multline}
In \cref{s:well-posedness} we showed that the HDG discretization
\cref{eq:hdg-discretization} is uniformly well-posed in the norm
induced by the inner-product on the right hand side of
\cref{eq:Pfirst}. From \cref{ss:mardal-winther} it therefore follows
that this preconditioner is a parameter-robust preconditioner for the
HDG discretization \cref{eq:hdg-discretization} before static
condensation. We then showed in
\cref{ss:Static-condensation,ss:preconds,sec:aux-problem,ss:condensed-preconditioner}
that the Schur complement $S_P$ of $P$ is a parameter-robust
preconditioner for \cref{eq:condensed-formulation}, i.e., for the HDG
discretization \emph{after} static condensation.

An alternative preconditioner $\widehat{P}$ for the HDG discretization
\cref{eq:hdg-discretization} is defined by:
\begin{multline*}
  \langle \widehat{P} \boldsymbol{x}_h, \boldsymbol{y}_h \rangle_{\boldsymbol{X}_h^*, \boldsymbol{X}_h}
  = d_h^s(\boldsymbol{u}_h^s, \boldsymbol{v}_h^s)
  + \alpha \mu \kappa^{-1/2} \langle \bar{u}_h^t, \bar{v}_h^t\rangle_{\Gamma^I}
  + (u_h^d, v_h^d)_{v,d}
  \\
  + (\boldsymbol{p}_h^s, \boldsymbol{q}_h^s)_{q,s}
  + \Tilde{d}_h(\boldsymbol{p}_h^d, \boldsymbol{q}_h^d)
  + \alpha^{-1} \mu^{-1} \kappa^{1/2}  \langle \bar{p}_h^d, \bar{q}_h^d \rangle_{\Gamma^I}.
\end{multline*}
Using the boundedness and coercivity of the operator $d_h^s$ we have
norm equivalence between the norm induced by $d_h^s$ and
$\tnorm{\cdot}_{v,s}$.  Likewise, we have equivalence between the norm
induced by $\Tilde{d}_h$ and $\tnorm{\cdot}_{q,d}$, see
\cref{eq:coer-bound-tilde-a}. Then, by \cref{eq:schur-comp-equiv-2},
we note that the Schur complement $S_{\widehat{P}}$ of $\widehat{P}$
defines a norm equivalent to $\tnorm{\cdot}_{\bar{X}_h}$. A
consequence is that $S_{\widehat{P}}^{-1}$ is also a parameter-robust
preconditioner for the condensed HDG discretization given by
\cref{eq:condensed-formulation}.

\section{Numerical examples}
\label{s:numex}

The HDG discretization of the coupled Stokes--Darcy problem and the
preconditioners $S_P$ and $S_{\widehat{P}}$ have been implemented in
NGSolve \cite{schoberl2014c++}. The unstructured meshes used in the
examples are generated by NETGEN \cite{schoberl1997netgen}. In this
section we consider the performance of $S_P$ and $S_{\widehat{P}}$ as
preconditioners to the NGSolve implementation of MINRES and GMRES. In
particular, for different test cases we will determine the number of
iterations needed for preconditioned MINRES and GMRES to reach a
relative preconditioned residual tolerance of $10^{-8}$.

The preconditioners $S_P$ and $S_{\widehat{P}}$ require inverses of
block diagonal matrices (block matrices associated to the trace
velocity degrees of freedom of the Stokes region, the trace pressure
degrees of freedom in the Stokes region, and the trace pressure
degrees of freedom in the Darcy region). When using direct solvers for
these blocks, we will refer to $S_P$ and $S_{\widehat{P}}$ as being
\emph{exact}. When using iterative solvers for these blocks, we will
refer to $S_P$ and $S_{\widehat{P}}$ as being \emph{inexact}. When
using iterative solvers, we use auxiliary space preconditioners (ASP)
\cite{fu2021uniform,fu2023uniform} with one backward Gauss--Seidel
smoothing step in combination with Hypre's BoomerAMG
\cite{yang2002boomeramg} for the trace velocity block of the Stokes
region and the trace pressure block of the Darcy region. A direct
solver is used for the trace pressure block of the Stokes region since
this is a face mass matrix.

\subsection{Manufactured solutions}
\label{ss:manufactured-sol}

In this experiment we consider $\Omega = [0,1]^{2}$, where
$\Omega^s = [0,1] \times [0.5, 1]$ and
$\Omega^d = [0,1] \times [0, 0.5]$. As manufactured solution we choose
\begin{align*}
  u|_{\Omega^s} 
  & = \begin{bmatrix}
    - (2 \pi^2)^{-1} \sin(\pi x_1) e^{x_2/2}
    \\
    \pi^{-1} \cos(\pi x_1) e^{x_2/2}
  \end{bmatrix}, 
  && p|_{\Omega^s} 
     = - \pi^{-1} \mu \kappa^{-1} \cos(\pi x_1) e^{x_2/2},
  \\
  u|_{\Omega^d} 
  & = \begin{bmatrix}
    - 2 \sin(\pi x_1) e^{x_2/2}
    \\
    \pi^{-1} \cos(\pi x_1) e^{x_2/2}
  \end{bmatrix},
  && p|_{\Omega^d} 
     = - 2 \pi^{-1} \mu \kappa^{-1} \cos(\pi x_1) e^{x_2/2}.
\end{align*}
Source terms an boundary conditions are determined accordingly.

Let $k=2$, $\alpha = 1$, $\mu = 1$, and $\kappa = 1$. In
\Cref{tab:manufactured-h} we list the number of iterations for
preconditioned MINRES and GMRES to reach convergence for varying
$h$. We observe that both preconditioners $S_P^{-1}$ and
$S_{\widehat{P}}^{-1}$, exact and inexact versions, are robust with
respect to mesh size $h$. We further observe that
$S_{\widehat{P}}^{-1}$ outperforms $S_P^{-1}$ and therefore choose to
only present results for $S_{\widehat{P}}^{-1}$ for the remaining
experiments.

In \Cref{tab:manufactured-parameter} we study the robustness of
preconditioner $S_{\widehat{P}}^{-1}$ with respect to the physical
parameters $\mu$ and $\kappa$. For this experiment, we use a mesh
consisting of 32768 simplices. We observe that both the exact and
inexact versions of the preconditioner are robust with respect to
these physical parameters.

\begin{table}[tbp]
  \centering
  \begin{tabular}{c|c|ccccc}
    \multicolumn{7}{c}{MINRES} \\
    \hline
    & & \multicolumn{5}{c}{$h$} \\
    \cline{3-7}
    & it. & $2^{-4}$ & $2^{-5}$ & $2^{-6}$ & $2^{-7}$ & $2^{-8}$ \\
    \hline
    \multirow{2}{*}{prec.}
    & $S_{P}^{-1}$                 & 136 (244) & 135 (242) & 133 (239) & 133 (232) & 130 (224) \\ [1ex]
    & $S_{\widehat{P}}^{-1}$   & 108 (221) & 106 (224) & 105 (220) & 103 (217) & 103 (209) \\ [1ex]
    \hline
    \multicolumn{7}{c}{GMRES} \\
    \hline
    & & \multicolumn{5}{c}{$h$} \\
    \cline{3-7}
    & it. & $2^{-4}$ & $2^{-5}$ & $2^{-6}$ & $2^{-7}$ & $2^{-8}$ \\
    \hline
    \multirow{2}{*}{prec.}
    & $S_{P}^{-1}$                 & 53 (76) & 53 (76) & 53 (76) & 53 (75) & 53 (75) \\ [1ex]
    & $S_{\widehat{P}}^{-1}$   & 57 (78) & 56 (77) & 56 (77) & 56 (76) & 56 (76)
  \end{tabular}
  \caption{$h$-robustness for the manufactured solutions case
    described in \cref{ss:manufactured-sol}. We report the number of
    MINRES and GMRES iterations required to reach convergence when
    using the preconditioners $S_{P}^{-1}$ and $S_{\widehat{P}}^{-1}$
    in both their exact and inexact forms. Iteration counts for the
    inexact forms are shown in parentheses.  }
  \label{tab:manufactured-h}
\end{table}

\begin{table}[tbp]
  \centering
  \begin{tabular}{c|c|ccccc}
    \multicolumn{7}{c}{MINRES} \\
    \hline
    & & \multicolumn{5}{c}{$\log_{10}(\mu)$} \\
    \cline{3-7}
    & it. & $-4$ & $-2$ & $0$ & $2$ & $4$ \\
    \hline
    \multirow{5}{*}{$\log_{10}(\kappa)$}
    & 4  & 105 (188) & 95 (188) & 95 (188) & 95 (188) & 95 (188) \\ [1ex]
    & 2  & 103 (210) & 103 (210) & 103 (210) & 103 (210) & 103 (210) \\ [1ex]
    & 0  & 105 (220) & 105 (220) & 105 (220) & 105 (220) & 105 (220) \\ [1ex]
    & -2 & 114 (265) & 114 (265) & 114 (265) & 114 (265) & 114 (265) \\ [1ex]
    & -4 & 140 (238) & 140 (238) & 139 (238) & 139 (238) & 139 (238) \\ [1ex]
    \hline
    \multicolumn{7}{c}{GMRES} \\
    \hline
    & & \multicolumn{5}{c}{$\log_{10}(\mu)$} \\
    \cline{3-7}
    & it. & $-4$ & $-2$ & $0$ & $2$ & $4$ \\
    \hline
    \multirow{5}{*}{$\log_{10}(\kappa)$}
    & 4  & 49 (64) & 56 (76) & 56 (77) & 56 (77) & 56 (77) \\ [1ex]
    & 2  & 49 (63) & 56 (76) & 56 (77) & 56 (77) & 56 (77) \\ [1ex]
    & 0  & 50 (64) & 56 (76) & 56 (77) & 56 (76) & 56 (76) \\ [1ex]
    & -2 & 47 (66) & 57 (82) & 63 (89) & 63 (89) & 63 (89) \\ [1ex]
    & -4 & 57 (81) & 71 (101) & 76 (113) & 76 (113) & 76 (113)    
  \end{tabular}
  \caption{$\mu$- and $\kappa$-robustness for the manufactured
    solutions test case described in \cref{ss:manufactured-sol}. We
    report the number of MINRES and GMRES iterations required to reach
    convergence when using the preconditioner $S_{\widehat{P}}^{-1}$
    in both its exact and inexact form. Iteration counts for the
    inexact forms are shown in parentheses.  }
  \label{tab:manufactured-parameter}
\end{table}

\subsection{Heterogeneous permeability case}
\label{ss:heterogeneous}

For this example we follow a similar configuration to the one
described in \cite[section 6.2]{cesmelioglu2020embedded}. We consider
the domain $\Omega = [0,1]^{2}$, with
$\Omega^s = [0,1] \times [0.5, 1]$ and
$\Omega^d = [0,1] \times [0, 0.5]$, where the boundary associated to
the Stokes region is given by $\Gamma^s = \Gamma^s_1 \cup \Gamma^s_2$,
where $\Gamma_1^s := \{x \in \Gamma^s:\ x_1 = 0 \text{ or } x_1 = 1\}$
and $\Gamma_2^s := \{x \in \Gamma^s:\ x_2 = 1\}$. The boundary for the
Darcy region is denoted by $\Gamma^d$. The boundary conditions are
given by $u = 0$ on $\Gamma_1^s$, $u = (x_2(3/2 - x_2)/5, 0)$ on
$\Gamma_2^s$, and $u \cdot n = 0$ on $\Gamma^d$. Next, the
heterogeneous permeability parameter is defined as follows:
\begin{equation*}
  \kappa(x_1, x_2)
  := 700 (1 + 0.5(\sin(10 \pi x_1) \cos(20 \pi x_2^2) 
  + \cos^2(6.4 \pi x_1) \sin(9.2 \pi x_2) )) + 100.
\end{equation*}
We use a fixed mesh composed of 32768 simplices. In
\Cref{fig:heterogeneousSolution}, we display the permeability,
velocity, and pressure fields corresponding to the configuration
described above with $k=2$, $\mu = 0.1$, $\alpha = 1$, $f^s = 0$, and
$f^d = 0$. In \Cref{tab:heterogeneous} we study the robustness of the
preconditioner $S_{\widehat{P}}^{-1}$ when combined with MINRES and
GMRES with respect to $h$ and $\mu$. We observe that the
preconditioner is parameter-robust in exact and inexact form.
 
\begin{figure}[tbp]
  \centering
  
  \begin{minipage}{0.48\textwidth}
    \centering
    \includegraphics[width=\textwidth]{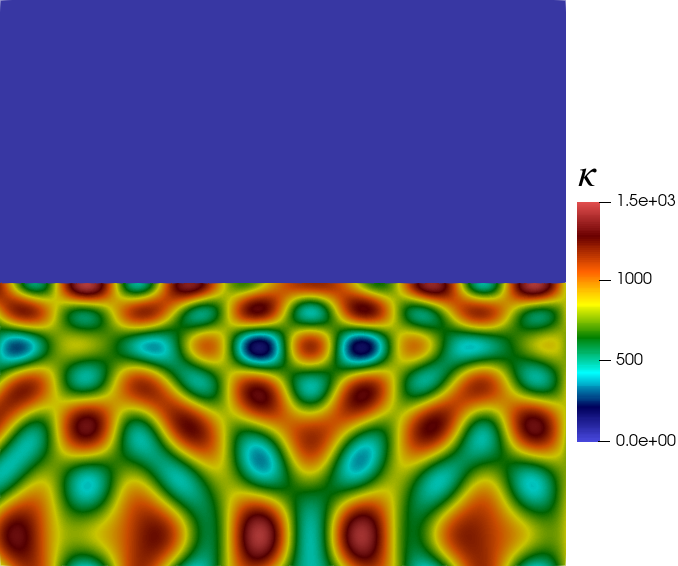}
  \end{minipage}
  \begin{minipage}{0.48\textwidth}
    \centering
    \includegraphics[width=\textwidth]{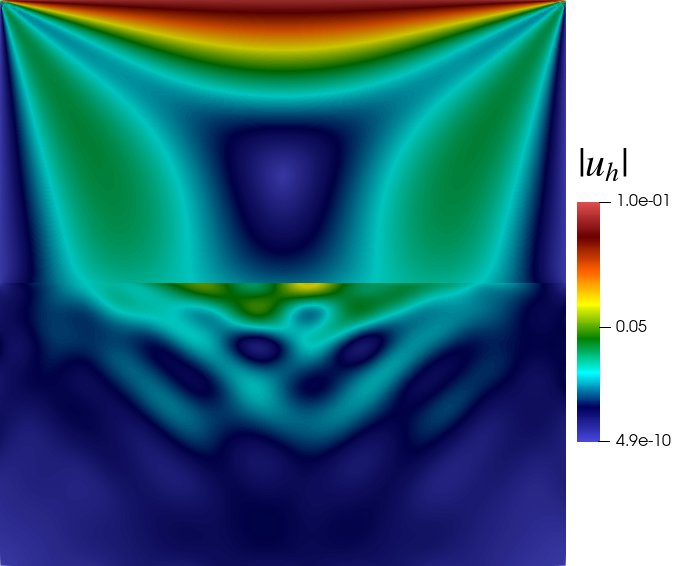}
  \end{minipage}
  
  \begin{minipage}{0.48\textwidth}
    \centering
    \includegraphics[width=\textwidth]{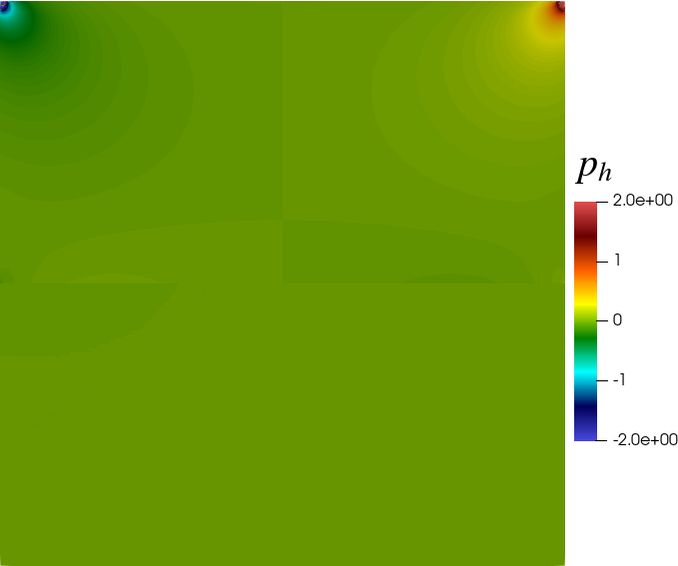}
  \end{minipage}
  
  \caption{Permeability, computed velocity field, and computed
    pressure field for the test case described in
    \cref{ss:heterogeneous}. The pressure lies between $-18$ and $18$
    but the scale is adjusted to $[-2,2]$ to emphasize the corner
    singularities of the pressure. }
  \label{fig:heterogeneousSolution}
\end{figure}

\begin{table}[tbp]
  \centering
  \begin{tabular}{c|c|ccccc}
    \multicolumn{7}{c}{MINRES} \\
    \hline
    & & \multicolumn{5}{c}{$h$} \\
    \cline{3-7}
    & it. & $2^{-4}$ & $2^{-5}$ & $2^{-6}$ & $2^{-7}$ & $2^{-8}$ \\
    \hline
    \multirow{5}{*}{$\log_{10}(\mu)$}
    & 0  & 107 (209) & 109 (217) & 109 (226) & 104 (222) & 97 (209) \\[1ex]
    & -2 & 107 (209) & 109 (217) & 109 (226) & 104 (222) & 97 (209) \\[1ex]
    & -4 & 107 (209) & 109 (217) & 109 (226) & 104 (222) & 97 (209) \\[1ex]
    & -6 & 107 (209) & 109 (217) & 109 (226) & 104 (222) & 97 (209) \\[1ex]
    \hline
    \multicolumn{7}{c}{GMRES} \\
    \hline
    & & \multicolumn{5}{c}{$h$} \\
    \cline{3-7}
    & it. & $2^{-4}$ & $2^{-5}$ & $2^{-6}$ & $2^{-7}$ & $2^{-8}$ \\
    \hline
    \multirow{5}{*}{$\log_{10}(\mu)$}
    & 0  & 57 (78) & 57 (80) & 56 (81) & 55 (80) & 55 (79) \\ [1ex]
    & -2 & 56 (73) & 56 (76) & 56 (78) & 55 (78) & 55 (78) \\ [1ex]
    & -4 & 51 (68) & 50 (69) & 49 (68) & 48 (66) & 48 (65) \\ [1ex]
    & -6 & 51 (68) & 50 (69) & 49 (67) & 47 (65) & 47 (63)    
  \end{tabular}
  \caption{$h$- and $\mu$-robustness for the heterogeneous
    permeability test case described in \cref{ss:heterogeneous}. We
    report the number of MINRES and GMRES iterations to reach
    convergence when using the preconditioner $S_{\widehat{P}}^{-1}$
    in both its exact and inexact forms.  Iteration counts for the
    inexact forms are shown in parentheses.}
  \label{tab:heterogeneous}
\end{table}

\section{Conclusions}
\label{s:conclusions}

We have introduced parameter-robust preconditioners for an HDG
discretization of the coupled Stokes--Darcy model with
Beavers--Joseph--Saffman interface conditions. We first presented a
preconditioner for the non-condensed linear system using the
operator-preconditioning framework of
\cite{mardal2011preconditioning}. Then, using the framework of
\cite{henriquez2025parameter}, we proved conditions that ensure that
the Schur complement of the preconditioner for the non-condensed
problem is a parameter-robust preconditioner for the statically
condensed HDG discretization. The preconditioners presented here avoid
the use of fractional operators or the imposition of interface
conditions on the finite element space, thus simplifying their
implementation. Numerical experiments are presented to support the
analysis.

\subsubsection*{Acknowledgements}

MK acknowledges support from the Research Council of Norway (NFR)
grant \#303362 the European Research Council under grant 101141807
(aCleanBrain). SR is supported by a Discovery Grant (RGPIN-2023-03237)
from the Natural Sciences and Engineering Research Council of Canada.

\bibliographystyle{plain}
\bibliography{references}
\appendix
\section{Proof of inf-sup condition \cref{eq:inf-sup-bhd}}
\label{ap:proofinfsup}

We will prove that for any
$\boldsymbol{q}_h^d \in \boldsymbol{Q}_h^d$ there exists
$v_h \in V_h^d$ such that
$\tnorm{v_h}_{v,d} \le c_3^{-1} \tnorm{\boldsymbol{q}_h^d}_{q,d}$ and
$b_h^d(\boldsymbol{q}_h^d,v_h) =
\tnorm{\boldsymbol{q}_h^d}_{q,d}^2$. First observe that
\begin{align*}
  b_h^d(\boldsymbol{q}_h^d, v_h) 
  :=& (\nabla q_h^d, v_h)_{\mathcal{T}_h^d} - \langle q_h^d - \bar{q}_h^d, v_h \cdot n \rangle_{\partial \mathcal{T}_h^d},
  \\
  \tnorm{\boldsymbol{q}_h^d}_{q,d}^2
  :=& \mu^{-1}\kappa \norm[0]{\nabla q_h^d}_{\mathcal{T}_h^d}^2
      + \mu^{-1}\kappa\eta \norm[0]{h_K^{-1/2}(q_h^d - \bar{q}_h^d)}_{\partial \mathcal{T}_h^d}^2.
\end{align*}
Then, following the proof of \cite[Lemma 4.2]{cesmelioglu2024strongly}
(see also \cite[Lemma 6]{henriquez2025preconditioning}), given
$\boldsymbol{q}_h^d \in \boldsymbol{Q}_h^d$, we define a function
$\tilde{z}_h \in V_h^d$ using the local BDM degrees of freedom
\cite[Proposition 2.3.2]{boffi2013mixed} as follows:
\begin{align*}
  (\tilde{z}_h, w_h)_K 
  &= \mu^{-1} \kappa (\nabla q_h^d, w_h)_K 
  && \forall w_h \in \mathcal{N}_{k-2}(K), \forall K \in \mathcal{T}_h^d,
  \\
  \langle \tilde{z}_h \cdot n, \bar{w}_h \rangle_{\partial K}
  &= - \mu^{-1} \kappa \eta h_K^{-1} \langle q_h^d - \bar{q}_h^d, \bar{w}_h \rangle_{\partial K}
  && \forall \bar{w}_h \in \mathcal{R}_k(\partial K), \forall K \in \mathcal{T}_h^d,
\end{align*}
where $\mathcal{N}_{k-2}(K)$ is the N\'ed\'elec space and
$\mathcal{R}_k(\partial K) = \{ \bar{w}_h \in L^2(\partial K) :
\bar{w}_h|_F \in \mathbb{P}_k(F), \forall F \subset \partial K
\}$. Therefore, if $w_h = \nabla q_h^d$ and
$\bar{w}_h = q_h^d - \bar{q}_h^d$, we find for all
$K \in \mathcal{T}_h^d$:
\begin{align*}
  (\tilde{z}_h, \nabla q_h^d)_K 
  = \mu^{-1} \kappa \norm[0]{\nabla q_h^d}_K^2
  \quad \text{and} \quad
  \langle \tilde{z}_h \cdot n, q_h^d - \bar{q}_h^d \rangle_{\partial K}
  = - \mu^{-1} \kappa \eta \norm[0]{h_K^{-1/2}(q_h^d-\bar{q}_h^d)}_{\partial K}^2,
\end{align*}
and so
$b_h^d(\boldsymbol{q}_h^d, \tilde{z}_h) =
\tnorm{\boldsymbol{q}_h^d}_{q,d}^2$. We are left to show that
$\tnorm{\tilde{z}_h}_{v,d} \le c_3^{-1}
\tnorm{\boldsymbol{q}_h^d}_{q,d}$. Similar to proof of \cite[Lemma
4.4]{linke2018quasi}, there exists a uniform constant $c$ such that
\begin{align*}
  \mu\kappa^{-1} & \norm[0]{\tilde{z}_h}_K^2 + \mu\kappa^{-1}\eta^{-1} h_K\norm[0]{\tilde{z}_h \cdot n}_{\partial K}^2
  \\
  \le & c \sup_{\substack{w_h \in \mathcal{N}_{k-2}(K) \\ 
  \norm[0]{w_h}_{K} = 1}} |\mu^{1/2}\kappa^{-1/2} (\tilde{z}_h, w_h)_K|^2
  + c \sup_{\substack{\bar{w}_h \in R_k(\partial K) \\ \norm[0]{\bar{w}_h}_{\partial K}=1}} h_K\eta^{-1} |\mu^{1/2}\kappa^{-1/2}\langle \tilde{z}_h \cdot n, \bar{w}_h \rangle_{\partial K}|^2
  \\
  =& c\sup_{\substack{w_h \in \mathcal{N}_{k-2}(K) \\ 
  \norm[0]{w_h}_{K} = 1}} |\mu^{-1/2}\kappa^{1/2} (\nabla q_h^d, w_h)_K|^2
  + c\sup_{\substack{\bar{w}_h \in R_k(\partial K) \\ \norm[0]{\bar{w}_h}_{\partial K}=1}} h_K\eta^{-1}|\mu^{-1/2}\kappa^{1/2} \eta h_K^{-1} \langle q_h^d-\bar{q}_h^d, \bar{w}_h \rangle_{\partial K}|^2
  \\
  \leq& c\mu^{-1}\kappa \norm[0]{\nabla q_h^d}_K^2 + c\mu^{-1}\kappa \eta \norm[0]{h_K^{-1/2}(q_h^d-\bar{q}_h^d)}_{\partial K}^2.
\end{align*}
It follows that there exists a uniform constant $c_3$ such that
$\tnorm{\tilde{z}_h}_{v,d} \le c_3^{-1}
\tnorm{\boldsymbol{q}_h^d}_{q,d}$.

\section{Auxiliary results}

\begin{lemma}
  \label{lem:facet-norm-bound}
  There exists a uniform $\bar{c}$ such that
  \begin{equation}
    \label{eq:facet-norm-bound}
    (2 \mu \eta)^{1/2} \tnorm{\bar{v}_h}_{v,h}
    \leq \bar{c} \tnorm{\boldsymbol{v}_h}_{v,s}
    \quad \forall \boldsymbol{v}_h \in \boldsymbol{V}_h.
  \end{equation}
\end{lemma}
\begin{proof}
  The proof follows the same steps as in the proof of \cite[Lemma
  5]{rhebergen2018preconditioning}. \qed
\end{proof}

\begin{lemma}
  \label{lem:interface-norm-bound}
  There exists a uniform constant $c_\Gamma > 0$ such that 
  \begin{align}
    \label{eq:interface-norm-bound-2}
    \norm[0]{\bar{v}_h}_{\Gamma^I} 
    & \leq c_\Gamma \big( \norm[0]{\varepsilon(v_h)}_{\mathcal{T}_h^s}^2 
    + \eta \norm[0]{h_K^{-1/2} (v_h - \bar{v}_h))}_{\partial \mathcal{T}_h^s}^2 \big)^{1/2},\\
    \label{eq:interface-norm-bound-4}
    \norm[0]{\bar{q}_h}_{\Gamma^I} 
    & \leq c_\Gamma \big( \norm[0]{\nabla q_h}_{\mathcal{T}_h^d}^2 
    + \eta \norm[0]{h_K^{-1/2} (q_h - \bar{q}_h))}_{\partial \mathcal{T}_h^d}^2 \big)^{1/2}, 
  \end{align}
  for all $(\boldsymbol{v}_h, \boldsymbol{q}_h) \in \boldsymbol{V}_h
  \times \boldsymbol{Q}_h$.
\end{lemma}
\begin{proof}
  We recall from \cite[eq. (14)]{cesmelioglu2023hybridizable2},
  \cite[Theorem 4.4]{girault2009dg}, and Korn's inequality
  \cite{brenner2004korn} that
  \begin{equation*}
    \norm[0]{v_h}_{\Gamma^I} 
    \leq c_\Gamma \big( \norm[0]{\varepsilon(v_h)}_{\mathcal{T}_h^s}^2 
    + \eta \norm[0]{h_K^{-1/2} (v_h - \bar{v}_h))}_{\partial \mathcal{T}_h^s}^2 \big)^{1/2}.
  \end{equation*}
  \Cref{eq:interface-norm-bound-2} then follows from
  \begin{equation*}
    \norm[0]{\bar{v}_h}_{\Gamma^I}
    \leq \norm[0]{v_h - \bar{v}_h}_{\Gamma^I} 
    + \norm[0]{v_h}_{\Gamma^I}
    \leq \eta^{1/2} \norm[0]{h_K^{-1/2} (v_h - \bar{v}_h)}_{\partial \mathcal{T}_h^s} 
    + \norm[0]{v_h}_{\Gamma^I}.
  \end{equation*}
  The proof for \cref{eq:interface-norm-bound-4} is the same, except
  using \cite[eq. (34)]{chaabane2017convergence}, i.e., using
  \begin{equation*}
    \norm[0]{q_h}_{\Gamma^I} 
    \leq c_\Gamma \big( \norm[0]{\nabla q_h}_{\mathcal{T}_h^d}^2 
    + \eta \norm[0]{h_K^{-1/2} (q_h - \bar{q}_h))}_{\partial \mathcal{T}_h^d}^2 \big)^{1/2}.
  \end{equation*}
  \qed
\end{proof}

Finally, we recall from \cite[Theorem 3]{henriquez2026robust} that if
$A$ is an operator with block structure \cref{eq:abstract-Ax=F-2}, $P$
is an operator with block structure \cref{eq:gnral-precond-matrix},
and if there exist uniform constants $c_1,c_2$ such that
\begin{equation}
  \label{eq:schur-comp-equiv-1}
  c_1 \norm[0]{\boldsymbol{x}_h}_{\boldsymbol{X}_h}^2
  \leq a_h(\boldsymbol{x}_h , \boldsymbol{x}_h)
  \leq c_2 \norm[0]{\boldsymbol{x}_h}_{\boldsymbol{X}_h}^2
  \quad \forall \boldsymbol{x}_h \in \boldsymbol{X}_h,
\end{equation}
then,
\begin{equation}
  \label{eq:schur-comp-equiv-2}
  c_1 \norm[0]{\bar{x}_h}_{\bar{X}_h}^2
  \leq \langle S_A \bar{x}_h, \bar{x}_h \rangle_{\bar{X}_h^*, \bar{X}_h}
  \leq c_2 \norm[0]{\bar{x}_h}_{\bar{X}_h}^2
  \quad \forall \bar{x}_h \in \bar{X}_h.
\end{equation}

\end{document}